\documentclass[12pt,a4paper]{article}
\pdfoutput=1 

\usepackage{amsmath,amssymb,amsthm}
\usepackage{apropos} 
\usepackage{geometry,graphicx,pict2e,authblk}
\usepackage[numbers,sort&compress]{natbib}

\newgeometry{left=2cm,right=2cm,top=3cm,bottom=3cm} 

\newtheorem{theorem}{Theorem}[section]
\newtheorem{lemma}{Lemma}[section]
\newtheorem{remark}{Remark}[section]

\title{Analysis of the Discontinuous Petrov-Galerkin Method with Optimal Test Functions for the Reissner-Mindlin Plate Bending Model}

\author{Victor M. Calo}
\author{Nathaniel O. Collier}
\affil{King Abdullah University of Science and Technology (KAUST) \\
Center for Numerical Porous Media \\
Thuwal, Kingdom of Saudi Arabia \\
victor.calo@kaust.edu.sa, nathaniel.collier@kaust.edu.sa} 
\author{Antti H. Niemi}
\affil{Aalto University \\
School of Engineering \\
Department of Civil and Structural Engineering \\
Espoo, Finland \\
antti.h.niemi@aalto.fi}
\date{}

\begin{document}

\maketitle

\begin{abstract}
We analyze the discontinuous Petrov-Galerkin (DPG) method with optimal test functions when applied to solve the Reissner-Mindlin model of plate bending. We prove that the hybrid variational formulation underlying the DPG method is well-posed (stable) with a thickness-dependent constant in a norm encompassing the $L_2$-norms of the bending moment, the shear force, the transverse deflection and the rotation vector. We then construct a numerical solution scheme based on quadrilateral scalar and vector finite elements of degree $p$. We show that for affine meshes the discretization inherits the stability of the continuous formulation provided that the optimal test functions are approximated by polynomials of degree $p+3$. We prove a theoretical error estimate in terms of the mesh size $h$ and polynomial degree $p$ and demonstrate numerical convergence on affine as well as non-affine mesh sequences.
\end{abstract}

\emph{Keywords}: plate bending; finite element method; discontinuous Petrov-Galerkin; discrete stability; optimal test functions; error estimates

\section{Introduction}
Finite element methods based on the principle of virtual displacements are the most widely used tools for computing the deformations and stresses of elastic bodies under external loads. However, in the modelling of thin-walled structures, the basic formulation leads to so-called locking, or numerical over-stiffness, unless special techniques (reduced integration, nonconforming elements) are applied, see~\cite{Macneal1978,Hughes1981,Belytschko1983,Bathe1985}. Another difficulty related to the displacement based formulations is the stress recovery. It is well known that the accuracy of the stress field derived from the displacement field can be much lower than that of the displacement field. Therefore special recovery techniques are often applied to improve the accuracy of stress approximations, see \cite{Szabo2009,Niemi2012}. Practical finite element design relies heavily on heuristics, intuition, and engineering expertise which make numerical analysis of the formulations difficult, since the various physical and geometrical assumptions do not have obvious interpretations in the functional analytic setting required for mathematical error analysis.

Mixed formulations where stresses are declared as independent unknowns are attractive because they often avoid the problem of locking by construction and allow direct approximation of the quantities of interest. However, in contrast to pure displacement formulations, mixed finite element methods do not inherit stability from the continuous formulation, but the stability of the discretization must be independently verified for each particular choice of finite element spaces as in \cite{Arnold1981,Pitkaranta1988,Arnold1989,Brezzi1989,Pitkaranta1996,Chapelle1998,Pitkaranta2000,Arnold2005a,Chinosi2006}. 
The recently introduced discontinuous Petrov-Galerkin (DPG) variational framework provides means for automatic computation of test functions that guarantee discrete stability for any choice of trial functions, see \cite{Demkowicz2010,DemGopNUMPDE2010,Demkowicz2011,Niemi2011b,DemGopSJNA2011,Demkowicz2012,Niemi2011,ZitMugDemGopParCalJCP2011,Niemi2012a}.   

In this paper we provide an error analysis for the DPG method with optimal test functions when applied to the Reissner-Mindlin model of plate bending. We follow the error analysis program laid down in \cite{DemGopSJNA2011,Gopalakrishnan2011}. The stability analysis utilizes a duality argument based on the concept of optimal test space norm and is better suited to multidimensional problems than the earlier (see \cite{DemGopNUMPDE2010,Demkowicz2011,Niemi2011b,Demkowicz2012}) analysis technique based on deriving an explicit expression for the generalized energy norm. 

The unknowns in the (mesh-dependent) DPG formulation of the Reissner-Mindlin model are the shear force, bending moment, transverse deflection and rotation (field variables) as well as their suitable traces defined independently on the mesh skeleton. First, we show that the well-posedness and stability of the ideal DPG variational formulation follows from the well-posedness of the bending-moment formulation of the Reissner-Mindlin model which was established in \cite{Beiraodaveiga2012}, see also \cite{Amara2002,Behrens2010,Behrens2011}. We introduce then a quadrilateral finite element discretization where the field variables are approximated by piecewise polynomial functions of degree $p$ or $p+1$ on each element and the traces by piecewise polynomials of degree $p$ (resultant tractions) and $p+1$ (displacements) on the mesh skeleton. We prove that on affine meshes the discrete formulation is stable in the sense of Babu\v{s}ka and Brezzi provided that the optimal test functions are approximated by piecewise polynomials of degree $p+3$ on each element.

The stability estimate is derived using regular (mesh and thickness independent) Sobolev norms and the estimate breaks down at the Kirchhoff limit corresponding to vanishing shear strains. Our final error bounds are therefore inversely proportional to the slenderness of the plate. The analysis indicates that the slenderness dependency arises from the shear stress term. This observation is corroborated by the numerical experiments which reveal that the accuracy of the shear stress is indeed affected by the value of the thickness while the other quantities are rather independent of it.

The paper is structured as follows. The derivation of the hybrid ultra-weak variational formulation of the Reissner-Mindlin plate bending model is presented in the next Section. The wellposedness of the formulation is proved in Section \ref{sec:wellposedness}. The corresponding finite element method is introduced and analyzed in Section \ref{sec:approximate problem} and the results of our numerical experiments are shown in Section \ref{sec:numerical results}. The paper ends with conclusions and suggestions for future work in Section \ref{sec:conclusions}.

\section{Reissner-Mindlin Plate Bending Model} \label{sec:variational form}
\subsection{Strong Form}
Let $\Omega$ be a convex polygonal domain in $\Reals^2$ representing the middle surface of a plate. We take $L = \mathrm{diam}(\Omega)$ as the length unit and assume that the plate thickness $t$ is small as compared with unity, that is the plate is thin. In the Reissner-Mindlin model, the deformation of the plate is described in terms of the transverse deflection $w$ and the rotation vector $\vpsi$, both defined on the middle surface $\Omega$. In the case of linearly elastic, homogeneous, and isotropic material, the shear force vector $\mV$ and the bending moment tensor $\mM$ are related to the displacements as (see for instance \cite{Ventsel2001})
\begin{equation} \label{eqn:constitutive laws}
\mV = \kappa G t(\vnabla w - \vpsi), \quad \mM = D t^3 \left[(1-\nu)\mepsilon(\vpsi) + \nu \trace(\mepsilon(\vpsi)) \eye \right],
\end{equation}
where $\eye$ is the identity tensor and $\mepsilon(\vpsi) = \frac{1}{2}(\vnabla\vpsi+\vnabla\vpsi^T)$ denotes the symmetric gradient. Moreover, 
\[
G = \frac{E}{2(1+\nu)}, \quad D = \frac{E}{12(1-\nu^2)}
\]
are the elastic material parameters written in terms of Young's modulus $E$ and Poisson's ratio $\nu$ while $\kappa >0$ is an additional model parameter called the shear correction factor. The fundamental balance laws of static equilibrium are
\begin{equation} \label{eqn:balance laws}
-\vnabla \cdot \mV = p, \quad - \vnabla \cdot \mM - \mV = \boldsymbol{0} ,
\end{equation}
where $p$ represents a transversal bending load. 

Upon rescaling the static quantities as 
\[
p \hookrightarrow Gt^3 p, \quad \mV \hookrightarrow Gt^3\mV, \quad \mM \hookrightarrow Gt^3 \mM,
\]
introducing the auxiliary variable $\vomega = \frac{1}{2}(\vnabla \vpsi - \vnabla \vpsi^T)$, and inverting the definition of $\mM$ in \eqref{eqn:constitutive laws} we arrive at the Reissner-Mindlin system
\begin{equation} \label{eqn:reissner-mindlin system}
\begin{alignedat}{2}
\kappa^{-1}t^2 \mV - \vnabla w + \vpsi &= \boldsymbol{0}, & \qquad
\cC^{-1} \mM - \vnabla \vpsi + \vomega &= \boldsymbol{0}, \\
- \vnabla \cdot \mV &= p, & \qquad
- \vnabla \cdot \mM - \mV &= \boldsymbol{0},
\end{alignedat}
\end{equation}
where
\[
\cC^{-1} \mtau= 6 \left(\mtau - \frac{\nu}{1+\nu} \trace(\mtau) \eye\right)
\]
is the two-dimensional ``compliance'' tensor.

\subsection{Hybrid Ultra-weak Form}
We use the usual Sobolev spaces $H^s(X)$ of scalar-valued functions defined on a domain $X \subset \Reals^2$ and boldface font for the vector- and tensor-valued analogues. As usual, $H^0(X) = L_2(X)$. Accordingly, we make use of the space $H(\mathrm{div},X)$ consisting of vector fields in $\bL_2(X)$ with divergence in $L_2(X)$ and denote by $\bH(\mathbf{div},X)$ the corresponding space of tensor-valued functions with rows in $H(\mathrm{div},X)$ (the divergence of a tensor is taken row-wise).

Let $\{\Omega_h\}$ be a non-degenerate family of partitions of $\Omega$ into convex quadrilaterals, where $h$ refers to the maximum element diameter in $\Omega_h$. Integration of the system \eqref{eqn:reissner-mindlin system} by parts over a single element $K$ in $\Omega_h$ gives
\begin{equation} \label{eqn:local ultra-weak form}
\begin{alignedat}{2}
\kappa^{-1}t^2 \scp{\mV}{\vq}_K + \scp{w}{\vnabla \cdot \vq}_K - \dual{w}{\vq \cdot \vn}_{\partial K} + \scp{\vpsi}{\vq}_K &= 0 & \quad &\forall \vq \in H(\mathrm{div},K) \\
\scp{\cC^{-1} \mM}{\mtau}_K + \scp{\vpsi}{\vnabla \cdot \mtau}_K - \dual{\vpsi}{\mtau \vn}_{\partial K} + \scp{r\mJ}{\mtau}_K &= 0 & \quad &\forall \, \mtau \in \boldsymbol{H}(\mathbf{div},K) \\
\scp{\mV}{\vnabla z}_K - \dual{z}{\mV \cdot \vn}_{\partial K} &= \scp{p}{z}_K & \quad &\forall z \in H^1(K) \\
\scp{\mM}{\vnabla \vphi}_K - \dual{\vphi}{\mM\vn}_{\partial K} - \scp{\mV}{\vphi}_K &= 0 & \quad &\forall \vphi \in \boldsymbol{H}^1(K) \\
\scp{\mM}{s\mJ}_K &= 0 & \quad &\forall s \in L_2(\Omega)
\end{alignedat}
\end{equation}
where $\vn$ denotes the outward unit normal on $\partial K$. The standard $L_2$ inner product of scalar-, vector- or tensor-valued functions over $K$ and $\partial K$ have been denoted by $\scp{\cdot}{\cdot}_K$ and $\dual{\cdot}{\cdot}_{\partial K}$, respectively. Moreover, the vorticity has been represented as a single unknown $\vomega = r\mJ$, where
\[
\mJ = \begin{bmatrix}
0 & 1 \\
-1 & 0
\end{bmatrix}
\]
and the equilibrium condition $M_{12}=M_{21}$ has been imposed weakly using the same notation.

The next step in developing the DPG formulation is to declare the traces $(w,\vpsi,\mV \cdot \vn,\mM\vn)\!\!\mid_{\partial K}$ as indepedent unknowns by rewriting the boundary terms as
\[
\begin{aligned}
\dual{w}{\vq \cdot \vn}_{\partial K} &\hookrightarrow \dual{\hat{w}}{\vq \cdot \vn}_{1/2,\partial K} \\
\dual{\vpsi}{\mtau\vn}_{\partial K} &\hookrightarrow \dual{\hat{\vpsi}}{\mtau \vn}_{1/2,\partial K} \\
\dual{z}{\mV \cdot \vn}_{\partial K} &\hookrightarrow \dual{z}{\hat{V}_n}_{1/2,\partial K} \\
\dual{\vphi}{\mM \vn}_{\partial K} &\hookrightarrow \dual{\vphi}{\hat{\mM}_n}_{1/2,\partial K} \\
\end{aligned}
\]
where $\dual{\cdot}{\ell}_{1/2,\partial K}$ denotes the action of a functional $\ell$ in $H^{-1/2}$ acting on scalar- or vector-valued functions.

The boundary conditions for a clamped boundary are $\hat{\vpsi} = \boldsymbol{0}$, $\hat{w} = 0$ on $\partial \Omega$ and the final variational form of the problem is obtained by summing \eqref{eqn:local ultra-weak form} over each $K$ in $\Omega_h$. The problem is to find $\vu = (\mV,\mM,w,\vpsi,r,\hat{w},\hat{\vpsi},\hat{V}_n,\hat{\mM}_n) \in \cbU$ such that
\begin{equation} \label{eqn:variational problem}
\cB(\vu,\vv) = \cL(\vv) \quad \forall \, \vv = (\vq,\mtau,z,\vphi,s) \in \cbV
\end{equation}
where the functional spaces are defined formally as
\begin{equation} \label{eqn:definition of spaces}
\begin{aligned}
\cbU &= \boldsymbol{L}_2(\Omega) \times \boldsymbol{L}_2(\Omega) \times L_2(\Omega) \times \boldsymbol{L}_2(\Omega) \times L_2(\Omega) \\
&\times H_0^{1/2}(\partial \Omega_h) \times \boldsymbol{H}^{1/2}_0(\partial \Omega_h) \times H^{-1/2}(\partial \Omega_h) \times \boldsymbol{H}^{-1/2}(\partial \Omega_h) \\
\cbV &= H(\mathrm{div},\Omega_h) \times \boldsymbol{H}(\mathbf{div},\Omega_h) \times H^1(\Omega_h) \times \boldsymbol{H}^1(\Omega_h) \times L_2(\Omega)
\end{aligned}
\end{equation}
and the bilinear and linear forms are given by
\begin{equation} \label{eqn:definition of forms}
\begin{aligned}
\cB(\vu,\vv) &= \scp{\mV}{\kappa^{-1}t^2 \vq + \vnabla z - \vphi}_{\Omega_h} + \scp{\mM}{\cC^{-1}\mtau + \vnabla \vphi +s\mJ}_{\Omega_h} \\
&+ \scp{w}{\vnabla \cdot \vq}_{\Omega_h} + \scp{\vpsi}{\vq + \vnabla \cdot \mtau}_{\Omega_h} + \scp{r \mJ}{\mtau}_{\Omega_h} - \dual{\hat{w}}{\vq \cdot \vn}_{\partial \Omega_h} - \dual{\hat{\vpsi}}{\mtau \vn}_{\partial \Omega_h} \\
&- \dual{z}{\hat{V}_n}_{\partial \Omega_h} - \dual{\vphi}{\hat{\mM}_n}_{\partial \Omega_h} \\
\cL(\vv) &= \scp{p}{z}_{\Omega_h}
\end{aligned}
\end{equation}
Here we have adopted the notation of \cite{Gopalakrishnan2011} for elementwise computations of the derivatives on the triangulation $\Omega_h$ and its skeleton $\partial \Omega_h$:
\[
\scp{\cdot}{\cdot}_{\Omega_h} = \sum_{K \in \Omega_h} \scp{\cdot}{\cdot}_K, \quad \dual{\cdot}{\cdot}_{\partial \Omega_h} = \sum_{K \in \Omega_h} \dual{\cdot}{\cdot}_{1/2,\partial K}
\]
The broken Sobolev spaces in \eqref{eqn:definition of spaces} are defined as
\[
\begin{aligned}
H^1(\Omega_h) &= \{ v \in L_2(\Omega) \; : \; v\!\!\mid_K \in H^1(K) \; \forall K \in \Omega_h \} \\
\bH(\mathrm{div},\Omega_h) &= \{ \vq \in \bL_2(\Omega) \; : \; \vq\!\!\mid_K \in \bH(\mathrm{div},K) \; \forall K \in \Omega_h \} \\
\end{aligned}
\]
whereas the fractional Sobolev spaces $H^{1/2}_0(\partial \Omega_h)$ and $H^{-1/2}(\partial \Omega_h)$ are interpreted as the trace spaces of functions in $H^1_0(\Omega)$ and $\bH(\mathrm{div},\Omega)$ on the skeleton $\partial \Omega_h$:
\[
\begin{aligned}
H^{1/2}_0(\partial\Omega_h) &= \{ v\!\!\mid_{\partial\Omega_h} \; : \; v \in H^1_0(\Omega) \} \\
H^{-1/2}(\partial\Omega_h) &= \{ \veta \cdot \vn \!\!\mid_{\partial\Omega_h} \; : \; \veta \in \bH(\mathrm{div},\Omega) \}
\end{aligned}
\]
The norms in the spaces $H^{1/2}_0(\partial \Omega_h)$ and $H^{-1/2}(\partial \Omega_h)$ can be defined as 
\[
\begin{aligned}
\norm{\hat{u}}_{H_0^{1/2}(\partial \Omega_h)} &= \inf_{v \in H^1_0(\Omega)} \{ \norm{v}_{H^1(\Omega)} : \gamma_0(v) = \hat{u} \} \\
\norm{\hat{\eta}_n}_{H^{-1/2}(\partial \Omega_h)} &= \inf_{\veta \in \bH(\mathrm{div},\Omega)} \{ \norm{\veta}_{\bH(\mathrm{div},\Omega)} : \vgamma_{\vn}(\veta) = \hat{\eta}_n \}
\end{aligned}
\]
where $\gamma_0$ and $\vgamma_{\vn}$ denote the trace operators satisfying $\gamma_0(v) = v\!\!\mid_{\partial \Omega_h}$ and $\vgamma_{\vn}(\veta) = \veta \cdot \vn\!\!\mid_{\partial \Omega_h}$ for all $v \in \cC^1(\bar{\Omega})$ and $\veta \in \boldsymbol{\cC}^1(\bar{\Omega})$, respectively.

\section{Well-posedness of the Ultra-Weak Formulation} \label{sec:wellposedness}
We begin with the following formulation of the Babu\v{s}ka-Lax-Milgram theorem and include the proof for completeness.
\begin{theorem} \label{theorem:babuska-lax-milgram}
Assume that $\cbU$ and $\cbV$ are two Hilbert spaces and $\cB(\vu,\vv)$ is a bilinear form on $\cbU \times \cbV$ satisfying
\begin{align}
\cB(\vu,\vv) &\leq C \norm{\vu}_{\cbU}\norm{\vv}_{\cbV} \quad \forall \vu \in \cbU, \vv \in \cbV 
\label{eqn:continuity of B} \\
\sup_{\vu \in \cbU} \frac{\cB(\vu,\vv)}{\norm{\vu}_{\cbU}} &\geq \alpha \norm{\vv}_{\cbV} \quad \forall \vv \in \cbV
\label{eqn:BB for the adjoint} \\
\cB(\vu,\vv) &= 0 \quad \forall \vv \in \cbV \quad \Rightarrow \quad \vu = \boldsymbol{0}
\label{eqn:injectivity}
\end{align}
If $\cL \in \cbV'$, that is $\cL$ is a linear functional on $\cbV$, there exists a unique $\vu \in \cbU$ such that 
\[
\cB(\vu,\vv) = \cL(\vv) \quad \forall \vv \in \cbV
\]
and
\[
\norm{\vu}_{\cbU} \leq \frac{\norm{\cL}}{\alpha}
\]
\end{theorem}
\begin{proof}
We show that the above assumptions guarantee that also the inf-sup condition
\begin{equation} \label{eqn:BB condition}
\sup_{\vv \in \cbV} \frac{\cB{(\vu,\vv)}}{\norm{\vv}_{\cbV}} \geq \alpha \norm{\vu}_{\cbU}  \quad \forall \vu \in \cbU
\end{equation}
holds. The assertion follows then from the Babu\v{s}ka-Lax-Milgram Theorem, see \cite[Theorem 2.1]{BabNUMERMATH1971}. To prove \eqref{eqn:BB condition} we define $\mT: \cbU \rightarrow \cbV$ and $\mT^*: \cbV \rightarrow \cbU$ through
\[
\cB(\vu,\vv) = \scp{\mT\vu}{\vv}_{\cbV} = \scp{\vu}{\mT^*\vv}_{\cbU}
\]
It follows from \eqref{eqn:continuity of B} and the Riesz Representation Theorem that $\mT$ and $\mT^*$ are continuous and that \eqref{eqn:BB for the adjoint} is equivalent to
\begin{equation} \label{eqn:BB of the adjoint 2}
\norm{\mT^* \vv}_{\cbU} \geq \alpha \norm{\vv}_{\cbV} \quad \forall \vv \in \cbV
\end{equation}
We show next that the range of $\mT^*$ is closed. Namely, if $\{\mT^* \vv_n\} \in \cbU$ is a Cauchy sequence, then so is $\{\vv_n\} \in \cbV$ because \eqref{eqn:BB for the adjoint} implies that
\[
\norm{\vv_m-\vv_n}_{\cbV} \leq \alpha \norm{\mT^*(\vv_m-\vv_n)}_{\cbU} = \alpha \norm{\mT^*\vv_m - \mT^*\vv_n}_{\cbU}
\]
Therefore $\{\vv_n\}$ converges to some $\vv \in \cbV$. Because $\mT^*$ is continuous $\{\mT^*\vv_n\}$ converges to $\mT^*\vv$ which proves that $\overline{\mT^*(\cbV)} = \mT^*(\cbV)$.

The condition \eqref{eqn:injectivity} implies now that $\mT^*$ is surjective. If this was not true, there would exist a non-zero $\tilde{\vu} \in \cbU$ such $\cB(\tilde{\vu},\vv)=\scp{\tilde{\vu}}{\mT^*\vv} = 0$ for every $\vv \in \cbV$. However, this contradicts \eqref{eqn:injectivity} so that we must have $\mT^*(\cbV) = \cbU$ which together with \eqref{eqn:BB of the adjoint 2} implies \eqref{eqn:BB condition}:
\[
\sup_{\vv \in \cbV} \frac{\cB(\vu,\vv)}{\norm{\vv}_{\cbV}} = \sup_{\vv \in \cbV} \frac{\scp{\vu}{\mT^*\vv}_{\cbU}}{\norm{\vv}_{\cbV}} \geq \sup_{\vv \in \cbV} \frac{\scp{\vu}{\mT^*\vv}_{\cbU}}{\alpha^{-1}\norm{\mT^*\vv}_{\cbU}} = \alpha \sup_{\vw \in \cbU} \frac{\scp{\vu}{\vw}_{\cbU}}{\norm{\vw}_{\cbU}} = \alpha \norm{\vu}_{\cbU}.
\]
\end{proof}

\subsection{Uniqueness of the Solution}
\begin{lemma} \label{lemma:uniqueness}
Let the spaces $\cbU,\cbV$ and the bilinear form $\cB(\vu,\vv)$ be as defined in Equations \eqref{eqn:definition of spaces} and \eqref{eqn:definition of forms}, respectively. If $\vu \in \cbU$ satisfies
\begin{equation} \label{eqn:zero-condition}
\cB(\vu,\vv) = 0
\end{equation}
for every $\vv \in \cbV$, then $\vu = \boldsymbol{0}$.
\end{lemma}
\begin{proof} 
Equation \eqref{eqn:zero-condition} implies that on every mesh element $K$ we have
\begin{equation} \label{eqn:weak system for injectivity}
\begin{alignedat}{2}
\kappa^{-1}t^2 \scp{\mV}{\vq}_K + \scp{w}{\vnabla \cdot \vq}_K - \dual{\hat{w}}{\vq \cdot \vn}_{\partial K} + \scp{\vpsi}{\vq}_K &= 0 & \quad &\forall \vq \in \bH(\mathrm{div},K) \\
\scp{\cC^{-1} \mM}{\mtau}_K + \scp{\vpsi}{\vnabla \cdot \mtau}_K - \dual{\hat{\vpsi}}{\mtau \vn}_{\partial K} + \scp{r\mJ}{\mtau}_K &= 0 & \quad &\forall \mtau \in \boldsymbol{H}(\mathbf{div},K) \\
\scp{\mV}{\vnabla z}_K - \dual{z}{\hat{V}_n}_{\partial K} &= 0 & \quad &\forall z \in H^1(K) \\
\scp{\mM}{\vnabla \vphi}_K - \dual{\vphi}{\hat{\mM}_n}_{\partial K} - \scp{\mV}{\vphi}_K &= 0 & \quad &\forall \vphi \in \boldsymbol{H}^1(K) \\
\scp{\mM}{s\mJ}_K &= 0 & \quad &\forall s \in L_2(K)
\end{alignedat}
\end{equation}
Testing with infinitely differentiable functions which are non-zero only on a compact subset of $K$ reveals that
\begin{equation} \label{eqn:distributional system for injectivity}
\begin{aligned}
\kappa^{-1}t^2 \mV - \vnabla w + \vpsi &= \boldsymbol{0} \\
\cC^{-1} \mM - \vnabla \vpsi + r \mJ &= \boldsymbol{0} \\
- \vnabla \cdot \mV &= 0 \\
- \vnabla \cdot \mM - \mV &= \boldsymbol{0}
\end{aligned}
\end{equation}
in every $K$ in the distributional sense. These equations in turn imply that $\mV \in \bH(\mathrm{div},K)$, $\mM \in \boldsymbol{H}(\mathbf{div},K)$ and $w \in H^1(K)$, $\vpsi \in \boldsymbol{H}^1(K)$.

We also have
\begin{equation} \label{eqn:compatibility of traces}
w\!\mid_{\partial K} = \hat{w}\!\mid_{\partial K},  \; \vpsi\!\mid_{\partial K}= \hat{\vpsi}\!\mid_{\partial K} 
\quad \text{and} \quad
\hat{V}_n\!\mid_{\partial K} = \mV \cdot \vn \!\mid_{\partial K},  \; \hat{\mM}_n\!\mid_{\partial K}= \mM\vn\!\mid_{\partial K} 
\end{equation}
This can be seen by integrating each equation in \eqref{eqn:weak system for injectivity} by parts and using the corresponding identity in \eqref{eqn:distributional system for injectivity} to show that
\begin{equation} \label{eqn:weak compatibility of traces}
\begin{alignedat}{2}
\dual{w-\hat{w}}{\vq \cdot \vn}_{1/2,\partial K} &= 0 &\quad &\forall \vq \in \bH(\mathrm{div},K) \\
\dual{\vpsi-\hat{\vpsi}}{\mtau \vn}_{1/2,\partial K} &= 0 &\quad &\forall \mtau \in \boldsymbol{H}(\mathbf{div},K) \\
\dual{z}{\mV\cdot \vn -\hat{V}_n}_{1/2,\partial K} &= 0 &\quad &\forall z \in H^1(K)  \\
\dual{\vphi}{\mM\vn-\hat{\mM}_n}_{1/2,\partial K} &= 0 &\quad &\forall \vphi \in \boldsymbol{H}^1(K)
\end{alignedat}
\end{equation}
These equations imply that $\mM \in \boldsymbol{H}(\mathbf{div},\Omega)$, $\mV \in \boldsymbol{H}(\mathrm{div},\Omega)$, and that $w \in H^1_0(\Omega)$, $\vpsi \in \boldsymbol{H}^1_0(\Omega)$ because $\hat{w}\!\mid_{\partial \Omega}=0$ and $\hat{\vpsi}\!\mid_{\partial \Omega} = \boldsymbol{0}$.

The extra regularity allows us to set $\mtau = \mM$, $\vq = \mV$ and $z = w$, $\vphi = \vpsi$ in \eqref{eqn:weak system for injectivity}. Summing the equations together and over every element, we find after integration by parts and simplification that
\begin{equation} \label{eqn:injectivity condition}
\kappa^{-1}t^2 \scp{\mV}{\mV}_{\Omega_h} + \dual{w-\hat{w}}{\mV \cdot \vn}_{\partial \Omega_h} + \scp{\cC^{-1}\mM}{\mM}_{\Omega_h} + \dual{\vpsi-\hat{\vpsi}}{\mM\vn}_{\partial \Omega_h} - \dual{w}{\hat{V}_n}_{\partial \Omega_h} - \dual{\vpsi}{\hat{\mM}_n}_{\partial \Omega_h} = 0
\end{equation}
The second and fourth terms vanish due to \eqref{eqn:weak compatibility of traces}. The last two terms vanish as well. To see this, we use~\eqref{eqn:compatibility of traces} and integrate by parts first locally and then globally (allowed by the regularity of $\mV,w,\mM,\vpsi$) to find that
\begin{equation} \label{eqn:vanishing jumps reasoning}
\begin{aligned}
\dual{w}{\hat{V}_n}_{\partial \Omega_h} &= \dual{w}{\mV \cdot \vn}_{\partial \Omega_h} \\
&=\scp{\vnabla w}{\mV}_{\Omega_h} - \scp{w}{\vnabla \cdot \mV}_{\Omega_h}  \\
&=\scp{\vnabla w}{\mV}_{\Omega} - \scp{w}{\vnabla \cdot \mV}_{\Omega} \\
&= \dual{w}{\mV \cdot \vn}_{\partial\Omega}
\end{aligned}
\end{equation}
Now the global boundary condition of $w \in H_0^1(\Omega)$ implies that $\dual{w}{\hat{V}_n}_{\partial \Omega_h} = 0$. A similar reasoning and the assumption $\vpsi = \boldsymbol{H}_0^1(\Omega)$ show that $\dual{\vpsi}{\hat{\mM}_n}_{\partial \Omega_h} = 0$.

Consequently, it follows from \eqref{eqn:injectivity condition} that $\mV$ and $\mM$ must be zero. To proceed further, we recall (see for example~\cite[Section VI]{Braess2001}) that for every $r \in L_2(\Omega)$, there exists a $\mtau^r \in \boldsymbol{H}(\mathbf{div},\Omega)$ such that $\vnabla \cdot \mtau^r = \boldsymbol{0}$ and $\tau^r_{12}-\tau^r_{21} = r$. We select $\mtau = \mtau^r$ in the second equation of \eqref{eqn:weak system for injectivity} and sum over the elements to conclude as in \eqref{eqn:vanishing jumps reasoning} that
\[
\begin{aligned}
\scp{r}{r}_{\Omega_h} &= \dual{\hat{\vpsi}}{\mtau^r\vn}_{\partial\Omega_h} =\dual{\vpsi}{\mtau^r\vn}_{\partial\Omega} =0
\end{aligned}
\]
Thus, $r=0$. 

Since $\mM$ and $r$ are already known to vanish, the second equation in \eqref{eqn:distributional system for injectivity} implies that $\vpsi$ is constant. Since $\vpsi \in \boldsymbol{H}_0^1(\Omega)$ we find that $\vpsi = \boldsymbol{0}$. The first equation in \eqref{eqn:distributional system for injectivity} implies then similarly that $w = 0$. Finally \eqref{eqn:compatibility of traces} shows that also the traces $\hat{w}$, $\hat{\vpsi}$ and $\hat{V}_n$, $\hat{\mM}_n$ are zero. Thus, all components in $\vu$ are shown to vanish and the proof is finished.

\end{proof}

\subsection{Existence of the Solution}
In the DPG terminology, the supremum in the condition \eqref{eqn:BB for the adjoint} is called the optimal test space norm:
\[
\enorm{\vv}_{\cbV} = \sup_{\vu \in \cbU} \frac{\cB(\vu,\vv)}{\norm{\vu}_{\cbU}}.
\]
In the current application it can be expressed in the form
\begin{equation}
\begin{aligned}
\enorm{\vv}_{\cbV}^2 &= \norm{\kappa^{-1}t^2 \vq + \vnabla z - \vphi}_{\Omega_h}^2 + \norm{\cC^{-1} \mtau + \vnabla \vphi + s\mJ}_{\Omega_h}^2 + \norm{\vnabla \cdot \vq}_{\Omega_h}^2 + \norm{\vq + \vnabla \cdot \mtau}_{\Omega_h}^2 \\ &
+ \norm{\tau_{12}-\tau_{21}}_{\Omega_h}^2
+ \norm{[\vq \cdot \vn]}_{\partial \Omega_h}^2 + \norm{[\mtau \vn]}_{\partial \Omega_h}^2 + \norm{[z\vn]}_{\partial \Omega_h}^2 + \norm{[\vphi \vn]}_{\partial \Omega_h}^2,
\end{aligned}
\end{equation}
where $\norm{\cdot}_{\Omega_h}^2 = \scp{\cdot}{\cdot}_{\Omega_h}$ and 
\[
\begin{alignedat}{2}
\norm{[\vq \cdot \vn]}_{\partial \Omega_h} &= \sup_{\hat{w} \in H^{1/2}_0(\partial \Omega_h)}\frac{\dual{\hat{w}}{\vq \cdot \vn}_{\partial\Omega_h}}{\norm{\hat{w}}_{H^{1/2}(\partial \Omega_h)}}, \quad & 
\norm{[\mtau \vn]}_{\partial \Omega_h} &= \sup_{\hat{\vpsi} \in \bH^{1/2}_0(\partial \Omega_h)}\frac{\dual{\hat{\vpsi}}{\mtau\vn}_{\partial\Omega_h}}{\norm{\hat{\vpsi}}_{\bH^{1/2}(\partial \Omega_h)}}, \\
\norm{[z\vn]}_{\partial \Omega_h} &= \sup_{\hat{V}_n \in H^{-1/2}(\partial \Omega_h)}\frac{\dual{z}{\hat{V}_n}_{\partial\Omega_h}}{\norm{\hat{V}_n}_{H^{-1/2}(\partial \Omega_h)}}, \quad & 
\norm{[\vphi \vn]}_{\partial \Omega_h} &= \sup_{\hat{\mM}_n \in \bH^{-1/2}(\partial \Omega_h)}\frac{\dual{\vphi}{\hat{\mM}_n}_{\partial\Omega_h}}{\norm{\hat{\mM}_n}_{\bH^{-1/2}(\partial \Omega_h)}}.
\end{alignedat}
\]
It is easy to see that conditions \eqref{eqn:BB for the adjoint} and \eqref{eqn:continuity of B} of the Babu\v{s}ka-Lax-Milgram Theorem are equivalent to the following Lemma.
\begin{lemma} \label{lemma:norm equivalence}
There exist positive constants $\alpha$ and $C$, which are independent of the mesh $\Omega_h$, such that
\begin{equation} \label{eqn:norm equivalence}
\alpha \norm{\vv}_{\cbV} \leq \enorm{\vv}_{\cbV} \leq C \norm{\vv}_{\cbV} \quad \forall \vv \in \cbV.
\end{equation}
\end{lemma}
\begin{proof}
Let $\vv = (\vq,\mtau,z,\vphi,s) \in \cbV$ be given and denote by 
\[
(\mV,\mM,w,\vpsi,r) \in \boldsymbol{H}(\mathrm{div},\Omega) \times \boldsymbol{H}(\mathbf{div},\Omega) \times L_2(\Omega) \times \boldsymbol{L}_2(\Omega) \times L_2(\Omega)
\]
the solution to the variational problem
\begin{equation} \label{eqn:weak system for existence}
\begin{alignedat}{2}
 \kappa^{-1}t^2\scp{\mV}{\delta \mV}_{\Omega}   + \scp{w}{\vnabla \cdot \delta \mV}_\Omega +  \scp{\vpsi}{\delta \mV}_\Omega &= \scp{\vq}{\delta \mV}_{\Omega} &\quad &\forall \delta\mV \in \boldsymbol{H}(\mathrm{div},\Omega), \\
 \scp{\cC^{-1} \mM}{\delta \mM}_{\Omega} + \scp{\vpsi}{\vnabla \cdot \delta \mM}_{\Omega} + \scp{r \mJ}{\delta \mM}_{\Omega} &= \scp{\mtau}{\delta\mM}_{\Omega} &\quad &\forall \delta\mM \in \boldsymbol{H}(\mathbf{div},\Omega), \\
 \scp{-\vnabla \cdot \mV}{\delta w}_{\Omega} &= \scp{z}{\delta w}_{\Omega} &\quad &\forall \delta w \in L_2(\Omega), \\
 \scp{-\vnabla \cdot \mM-\mV}{\delta \vpsi}_{\Omega} &= \scp{\vphi}{\delta \vpsi}_{\Omega} &\quad &\forall \delta \vpsi \in \boldsymbol{L}_2(\Omega), \\
 \scp{\mM}{\delta r\mJ}_{\Omega} &= \scp{s}{\delta r}_{\Omega} &\quad &\forall \delta r \in L_2(\Omega),
\end{alignedat}
\end{equation}
which exists and is unique due to the wellposedness of the bending moment formulation of the Reissner-Mindlin model. Namely, the analysis of \cite{Beiraodaveiga2012} shows that the bilinear form induced by the left hand side of \eqref{eqn:weak system for existence} satisfies the inf-sup condition in a norm encompassing
\begin{equation} \label{eqn:norm components}
t \norm{\mV}_{\bL_2(\Omega)}, \; \norm{\vnabla \cdot \mV}_{\bL_2(\Omega)}, \; \norm{\mM}_{\bL_2(\Omega)}, \; \norm{\vnabla \cdot \mM + \mV}_{\bL_2(\Omega)}, \; \norm{w}_{L_2(\Omega)},\; \norm{\vpsi}_{\bL_2(\Omega)},\; \norm{r}_{L_2(\Omega)}.
\end{equation}
Testing with infinitely smooth functions in the first two equations of \eqref{eqn:weak system for existence} reveals then that $w \in H^1(\Omega)$, $\vpsi \in \bH^1(\Omega)$ so that the solution of \eqref{eqn:weak system for existence} satisfies the estimate
\begin{equation} \label{eqn:regularity estimate}
\begin{aligned}
t\norm{\mV}_{\bL_2(\Omega)} + \norm{\vnabla \cdot \mV}_{L_2(\Omega)} &+ 
\norm{\mM}_{\bL_2(\Omega)}
+ \norm{\vnabla \cdot \mM + \mV}_{\bL_2(\Omega)} + \norm{w}_{H^1(\Omega)} + \norm{\vpsi}_{\boldsymbol{H}^1(\Omega)} + \norm{r}_{\bL_2(\Omega)} \\
&\leq C\left(\norm{\vq}_{\bL_2(\Omega)} + \norm{\mtau}_{\bL_2(\Omega)} + \norm{z}_{L_2(\Omega)} +\norm{\vphi}_{\bL_2(\Omega)} + \norm{s}_{L_2(\Omega)}\right)
\end{aligned}
\end{equation}
where the constant $C>0$ is independent of $t$, $\vq$, $\mtau$, $z$, $\vphi$, and $s$. 

The passage from \eqref{eqn:weak system for injectivity} to \eqref{eqn:distributional system for injectivity} can be repeated to arrive from \eqref{eqn:weak system for existence} to the system
\begin{equation} \label{eqn:distributional system for existence}
\begin{aligned}
\kappa^{-1}t^2 \mV - \vnabla w + \vpsi &= \vq \\
\cC^{-1} \mM - \vnabla \vpsi + r \mJ &= \mtau \\
- \vnabla \cdot \mV &= z \\
- \vnabla \cdot \mM - \mV &= \vphi
\end{aligned}
\end{equation}
valid on each $K$ in the distributional sense. Now integration by parts yields
\[
\begin{aligned}
\norm{\vq}_{\bL_2(\Omega)}^2 + \norm{\mtau}_{\bL_2(\Omega)}^2 + \norm{z}_{L_2(\Omega)}^2 &+\norm{\vphi}_{\bL_2(\Omega)}^2 + \norm{s}_{L_2(\Omega)}^2 \\
&= \scp{\kappa^{-1}t^2 \mV - \vnabla w + \vpsi}{\vq}_{\Omega} + \scp{\cC^{-1}\mM - \vnabla \vpsi + r\mJ}{\mtau}_{\Omega} \\
&\quad - \scp{\vnabla \cdot \mV}{z}_{\Omega} - \scp{\vnabla \cdot \mM + \mV}{\vphi}_{\Omega} + \scp{\mM}{s \mJ}_{\Omega} \\
&= \kappa^{-1}t^2 \scp{\mV}{\vq}_{\Omega_h} + \scp{w}{\vnabla \cdot \vq}_{\Omega_h} - \dual{w}{\vq \cdot \vn}_{\partial\Omega_h} + \scp{\vpsi}{\vq}_{\Omega_h} \\
&\quad+\scp{\cC^{-1}\mM}{\mtau}_{\Omega_h} + \scp{\vpsi}{\vnabla \cdot \mtau}_{\Omega_h} - \dual{\vpsi}{\mtau\vn}_{\partial\Omega_h} + \scp{r\mJ}{\mtau}_{\Omega_h} \\
&\quad +\scp{\mV}{\vnabla z}_{\Omega_h} - \dual{z}{\mV \cdot \vn}_{\partial\Omega_h} \\
&\quad + \scp{\mM}{\vnabla \vphi}_{\Omega_h} - \dual{\vphi}{\mM\vn}_{\partial\Omega_h} - \scp{\mV}{\vphi}_{\Omega_h} + \scp{\mM}{s\mJ}_{\Omega_h}
\end{aligned}
\]
Collecting terms and applying Cauchy-Schwarz inequality, we get
\[
\begin{aligned}
\norm{\vq}_{\bL_2(\Omega)}^2 + \norm{\mtau}_{\bL_2(\Omega)}^2 &+ \norm{z}_{L_2(\Omega)}^2 +\norm{\vphi}_{\bL_2(\Omega)}^2 + \norm{s}_{L_2(\Omega)}^2 \\
&= \scp{\mV}{\kappa^{-1}t^2\vq + \vnabla z - \vphi}_{\Omega_h} + \scp{\mM}{\cC^{-1}\mtau + \vnabla \vphi + s\mJ}_{\Omega_h} \\
& \quad + \scp{w}{\vnabla \cdot \vq}_{\Omega_h} + \scp{\vpsi}{\vq + \vnabla \cdot \mtau}_{\Omega_h} + \scp{r\mJ}{\mtau}_{\Omega_h} \\
& \quad - \dual{w}{\vq \cdot \vn}_{\partial\Omega_h} - \dual{\vpsi}{\mtau\vn}_{\partial\Omega_h} - \dual{z}{\mV \cdot \vn}_{\partial\Omega_h} - \dual{\vphi}{\mM \vn}_{\partial\Omega_h} \\
&\leq \norm{\mV}_{\bL_2(\Omega)}\norm{\kappa^{-1}t^2\vq + \vnabla z - \vphi}_{\Omega_h} + \norm{\mM}_{\bL_2(\Omega)}\norm{\cC^{-1}\mtau + \vnabla \vphi + s\mJ}_{\Omega_h} \\
& \quad + \norm{w}_{L_2(\Omega)}\norm{\vnabla \cdot \vq}_{\Omega_h} + \norm{\vpsi}_{\bL_2(\Omega)}\norm{\vq + \vnabla \cdot \mtau}_{\Omega_h} + \norm{r}_{L_2(\Omega)}\norm{\tau_{12}-\tau_{21}}_{L_2(\Omega)} \\
& \quad + \norm{[\vq \cdot \vn]}_{\partial\Omega_h}\norm{w}_{H^1(\Omega)} + \norm{[\mtau\vn]}_{\partial\Omega_h}\norm{\vpsi}_{\boldsymbol{H}^1(\Omega)} \\
& \quad + \norm{[z]}_{\partial\Omega_h}\norm{\mV}_{\boldsymbol{H}(\mathrm{div},\Omega)} + \norm{[\vphi]}_{\partial\Omega_h}\norm{\mM}_{\boldsymbol{H}(\mathbf{div},\Omega)} \\
&\leq 2\enorm{\vv}_{\cbV}(\norm{\mV}_{\boldsymbol{H}(\mathrm{div},\Omega)} + \norm{\mM}_{\boldsymbol{H}(\mathbf{div},\Omega)}
+ \norm{w}_{{H}^1(\Omega)} + \norm{\vpsi}_{\boldsymbol{H}^1(\Omega)} + \norm{r}_{L_2(\Omega)})
\end{aligned}
\]
By using the estimate \eqref{eqn:regularity estimate}, we obtain
\[
\begin{aligned}
\norm{\vq}_{\bL_2(\Omega)}^2 + \norm{\mtau}_{\bL_2(\Omega)}^2 + &\norm{z}_{L_2(\Omega)}^2 + \norm{\vphi}_{\bL_2(\Omega)}^2 +  \norm{s}_{L_2(\Omega)}^2 \\
&\leq Ct^{-1} \enorm{\vv}_{\cbV}\left(\norm{\vq}_{\bL_2(\Omega)} + \norm{\mtau}_{\bL_2(\Omega)} + \norm{z}_{L_2(\Omega)} +\norm{\vphi}_{\bL_2(\Omega)} + \norm{s}_{L_2(\Omega)}\right)
\end{aligned}
\]
and, consequently,
\begin{equation} \label{eqn:L2 bounds}
\norm{\vq}_{\bL_2(\Omega)} + \norm{\mtau}_{\bL_2(\Omega)} + \norm{z}_{L_2(\Omega)} +\norm{\vphi}_{\bL_2(\Omega)} + \norm{s}_{L_2(\Omega)} \leq Ct^{-1} \enorm{\vv}_{\cV}.
\end{equation}

The remaining terms constituting the norm $\norm{\vv}_{\cbV}$ can be bounded from above by $\enorm{\vv}_{\cbV}$ directly or by using the triangle inequality:
\begin{equation} \label{eqn:derivative bounds}
\begin{aligned}
\norm{\vnabla \cdot \vq}_{\Omega_h} &\leq \enorm{\vv}_{\cbV} \\
\norm{\vnabla \cdot \mtau}_{\Omega_h} &\leq \norm{\vnabla \cdot \mtau + \vq}_{\Omega_h} + \norm{\vq}_{\bL_2(\Omega)} \leq Ct^{-1}\enorm{\vv}_{\cbV} \\
\norm{\vnabla z}_{\Omega_h} &\leq \norm{\kappa^{-1}t^2 \vq + \vnabla z - \vphi}_{\Omega_h} + \kappa^{-1}t^2 \norm{\vq}_{\bL_2(\Omega)} + \norm{\vphi}_{\bL_2(\Omega)} \leq Ct^{-1} \enorm{\vv}_{\cV} \\
\norm{\vnabla \vphi}_{\Omega_h} &\leq \norm{\cC^{-1}\mtau+\vnabla \vphi + s \mJ}_{\Omega_h} + \norm{\cC^{-1} \mtau}_{\bL_2(\Omega)} + 2\norm{s}_{L_2(\Omega)}^2 \leq Ct^{-1} \enorm{\vv}_{\cV}
\end{aligned}
\end{equation}
The first inequality in \eqref{eqn:norm equivalence} follows now from \eqref{eqn:derivative bounds} and \eqref{eqn:L2 bounds} with an $\alpha$ proportional to $t$.

The proof of the second inequality is more straightforward. The integral terms $\norm{\cdot}_{\Omega_h}$ can be bounded from above by $\norm{\vv}_{\cbV}$ using the triangle inequality whereas the jump terms can be handled by integration by parts and Cauchy-Schwarz inequality:
\[
\norm{[\vq \cdot \vn]}_{\partial\Omega_h} = \sup_{z \in H_0^1(\Omega)} \frac{\dual{z}{\vq \cdot \vn}_{\partial\Omega_h}}{\norm{z}_{H^1(\Omega)}} = \sup_{z \in H_0^1(\Omega)} \frac{\scp{\vnabla z}{\vq}_{\Omega_h} + \scp{z}{\vnabla \cdot \vq}_{\Omega_h}}{\norm{z}_{H^1(\Omega)}}
\leq \norm{\vq}_{\boldsymbol{H}(\mathrm{div},\Omega_h)}
\]
Similar arguments can be used to show that
\[
\begin{aligned}
\norm{[\mtau\vn]}_{\partial\Omega_h} &\leq \norm{\mtau}_{\boldsymbol{H}(\mathbf{div},\Omega_h)} \\
\norm{[z]}_{\partial\Omega_h} &\leq \norm{z}_{H_0^1(\Omega_h)} \\
\norm{[\vphi]}_{\partial\Omega_h} &\leq \norm{\vphi}_{\boldsymbol{H}_0^1(\Omega_h)}
\end{aligned}
\]
We leave the details to the reader and conclude our proof.
\end{proof}

We have shown in Lemmas \ref{lemma:uniqueness} and \ref{lemma:norm equivalence} that the conditions of the Babu\v{s}ka-Lax-Milgram theorem \ref{theorem:babuska-lax-milgram} hold. In other words, we have established
\begin{theorem}
The ultra-weak variational formulation of the Reissner-Mindlin plate bending problem defined by \eqref{eqn:variational problem}--\eqref{eqn:definition of forms} is well-posed.
\end{theorem}
\begin{remark}
The proportionality of $\alpha$ to $t$, in \eqref{eqn:norm equivalence}, is due to the first term in \eqref{eqn:regularity estimate} which affects only the shear stress. This observation is ratified in our numerical experiments below.
\end{remark}

\section{The Approximate Problem} \label{sec:approximate problem}
In order to discretize \eqref{eqn:variational problem}, we choose a finite element trial function space $\cbU_h \subset \cbU$ and construct a corresponding test function space $\cbV^r_h = \mT^r(\cbU_h) \subset \cbV^r \subset \cbV$ by solving the auxiliary problem
\[
\scp{\mT^{r}\vw_h}{\vv}_{\cbV} = \mathbf{B}(\vw_h,\vv) \quad \forall \vv \in \cbV^r
\]
for each $\vw_h \in \cbU_h$. The discontinuous Petrov-Galerkin approximation $\vu_h \in \cbU_h$ is defined as the solution to the problem
\begin{equation} \label{eqn:dpg approximation}
\cB(\vu_h,\vv) = \cL(\vv) \quad \forall \vv \in \cbV_h^r
\end{equation}
The space $\cbV^r$ is determined by an appropriate enrichment of the trial function space $\cbU_h$. The level of enrichment is specified so that the Fortin's Criterion for the discrete inf-sup condition holds:
\begin{lemma} \label{lemma:fortin}
(Fortin's Criterion for DPG)
Suppose that for the subspaces $\cbV^r$, $\cbU_h$, there exists a bounded linear projector $\mPi_h : \cbV \rightarrow \cbV^r$ such that
\begin{equation} \label{eqn:projection condition}
\cB(\vw_h,\vv-\mPi_h \vv) = 0 \quad \forall \vw_h \in \cbU_h.
\end{equation}
If $\norm{\mPi_h} \leq c$, then the finite element spaces $\cbU_h$ and $\cbV_h^r$ satisfy the inf-sup condition
\begin{equation}
\sup_{\vv_h^r \in \cbV_h^r} \frac{\mB(\vu_h,\vv_h^r)}{\norm{\vv_h^r}_{\cbV}} \geq \frac{\alpha}{c} \norm{\vu_h}_{\cbU} \quad \forall \vu_h \in \cbU_h
\end{equation}
and the DPG approximation is uniquely defined by \eqref{eqn:dpg approximation} and is a quasi-optimal approximation of $\vu$, namely
\begin{equation} \label{eqn:best approximation property}
\norm{\vu - \vu_h}_{\cbU} \leq \frac{Cc}{\alpha} \min_{\vw_h \in \cbU_h} \norm{\vu - \vw_h}_{\cbU}
\end{equation}
\end{lemma}
\begin{proof}
See proof of Theorem 2.1 in \cite{Gopalakrishnan2011}.
\end{proof}

To make Lemma \ref{lemma:fortin} applicable in the present context, we need to construct local projectors from $\boldsymbol{H}(\mathrm{div},K)$ and $H^1(K)$ to suitable finite element spaces. In \cite{Gopalakrishnan2011}, these projectors were constructed for polynomial spaces on simplicial triangulations of $\Omega$. We will use the techniques of \cite{Arnold2005} to construct analogous projectors for quadrilateral meshes. We assume the partitions to be shape-regular in the usual sense, that is, each angle of each $K \in \Omega_h$ is assumed to be bounded away from $0$ and $\pi$ by an absolute, positive constant and the ratio of any two sides on $K$ is assumed to be uniformly bounded. 

Let $\hat{K}$ be a rectangular reference element, and denote by $\mF_K:\hat{K} \rightarrow \mathbb{R}^2$ the bilinear diffeomorphism onto the actual element $K = \mF_K(\hat{K})$. We define the local bilinear quadrilateral finite element space of degree $r$ as
\begin{equation} \label{eqn:bilinear FE space of degree r}
S_r(K) = \{ v \in L_2(K), \; v = \hat{v} \circ \mF_K^{-1}, \; \hat{v} \in \cQ_r(\hat{K}) \},
\end{equation}
where $\cQ_r(\hat{K}) = \cP_{r,r}(\hat{K})$ denotes the space of polynomials of degree at most $r$ in each variable separately on $\hat{K}$. We also use the local vector finite element space
\begin{equation} \label{eqn:vector finite element space}
\bV\!\!_r(K) = \{ \vq:K\rightarrow \mathbb{R}^2 \; \mid \; \vq = (\mP_K \hat{\vq}) \circ \mF_K^{-1}, \; \hat{\vq} \in \cbR\cbT_{r}(\hat{K})\},
\end{equation}
where $\cbR\cbT_{r}(\hat{K})=\cP_{r+1,r}(\hat{K}) \times \cP_{r,r+1}(\hat{K})$ is the Raviart-Thomas space and $\mP_K$ denotes the Piola transformation which is defined in terms of the Jacobian matrix $\mJ_K = D\mF_K$ as
\[
\mP_K(\hat{\vx}) = \frac{\mJ_K(\hat{\vx})}{\det{\mJ_K(\hat{\vx})}}.
\]

For the numerical fluxes and traces we need local polynomial spaces defined on the boundary $\partial K$ as
\[
\begin{aligned}
\Gamma_r(\partial K) &= \{ \gamma \in L_2(\partial K), \; \gamma\!\!\mid_E \in \cP_r(E) \; \text{for all edges $E$ of $K$} \}, \\
\tilde{\Gamma}_r(\partial K) &= \Gamma_r(\partial K) \cap \cC(\partial K),
\end{aligned}
\]
where $\cP_r(\partial K)$ stands for polynomials of degree $r$ on $E$ and $\cC(\partial K)$ stands for the space of continuous functions on $\partial K$.

The trial space of degree $p$ for the DPG method is defined in terms of the above spaces\footnote{A tensor-valued function is included in $\bV_p(K)$ row-wise according to the definition \eqref{eqn:vector finite element space}.} as
\[
\begin{split}
\boldsymbol{\cbU}_h &= \{(\mV,\mM,w,\vpsi,r,\hat{w},\hat{\vpsi},\hat{V}_n,\hat{\mM}_n) \in \cbU \; : \\
&\mV\!\!\mid_K \in \bV_p(K), \; 
\mM\!\!\mid_K \in \bV_p(K), \; 
w\!\!\mid_K \in S_p(K),\;  
\vpsi\!\!\mid_K \in \bS_p(K), \; 
r\!\!\mid_K \in S_p(K), \\
 &\hat{w}\!\!\mid_{\partial K} \in \tilde{\Gamma}_{p+1}(\partial K), \;  
 \hat{\vpsi}\!\!\mid_{\partial K} \in \tilde{\Gamma}_{p+1}(\partial K), \; 
 \hat{V}_n\!\!\mid_{\partial K} \in \Gamma_{p}(\partial K), \;
 \hat{\mM}_n\!\!\mid_{\partial K} \in \Gamma_{p}(\partial K)  
 \quad \forall K \in \Omega_h
 \}
\end{split}
\]

In the definition of the enriched test function space $\cbV^r$, we may employ the space \eqref{eqn:bilinear FE space of degree r} to approximate those components which belong to $H^1(K)$ or $L_2(K)$ and the space~\eqref{eqn:vector finite element space} to approximate the components in $\bH(\mathrm{div},K)$. The definition of $\cbV^r$ is
\[
\begin{split}
\cbV^r = \{ &(\vq,\mtau,z,\vphi,s) \in \cbV \; : \;
\vq\!\!\mid_K \in \bV\!_r(K), \; 
\mtau\!\!\mid_K \in \bV\!_r(K), \\
&z\!\!\mid_K \in S_r(K), \;
\vphi\!\!\mid_K \in \bS_r(K), \;
\mu\!\!\mid_K \in S_r(K) \quad \forall K \in \Omega_h 
\}.
\end{split}
\]

Next we will show, that taking $r=p+3$ is sufficient to guarantee the existence of the projector needed to guarantee the best approximation property of $\vu_h$ in Lemma  \ref{lemma:fortin}. The proof consists of three parts and follows closely the reasoning used in \cite{Gopalakrishnan2011} with small modifications.

\begin{lemma}
Let $B(K)$ be defined as $B(K) = \{ z \in S_{p+2}(K) \; : \; \text{$z$ is zero at the vertices of $K$}\}$. Then there exists a projector $R_K^0$ onto $B(K) \subset H^1(K)$ such that
\begin{align}
\scp{R_K^0 z}{v}_K &= \scp{z}{v}_K \quad \forall v \in S_p(K) \label{eqn:R_K^0 first} \\ 
\dual{R_K^0 z}{\gamma}_{\partial K} &= \dual{z}{\gamma}_{\partial K} \quad \forall \gamma \in \Gamma_p( \partial K) \label{eqn:R_K^0 second} \\
h_K^{-1}\norm{R_K^0 z}_{L_2(K)} + \abs{R_K^0 z}_{H^1(K)} 
&\leq C(h_K^{-1}\norm{z}_{L_2(K)}+\abs{z}_{H^1(K)})
\label{eqn:R_K^0 third}
\end{align}
for all $z \in H^1(K)$.
\end{lemma}
\begin{proof}
To see that $R_K^0$ is well-defined, we first note that the number of conditions in \eqref{eqn:R_K^0 first} and \eqref{eqn:R_K^0 second} is
\[
\dim S_p(K) + \dim \Gamma_p(\partial K) = (p+1)^2 + 4(p+1) = p^2 +6p + 5
\]
and equals the dimension of $B(K)$:
\[
\dim B(K) = (p+3)^2 - 4 = p^2 +6p + 5
\]
Therefore, in order to show that $R_K^0 z$ exists and is unique, it suffices to show that $z = 0$ implies $R_K^0 z = 0$. On each edge $e$ of $\partial K$, $R_K^0 z$ has the form $R_K^0 z\!\!\mid_e = B_e u$ where $u \in \cP_p(e)$ and $B_e$ is a quadratic bubble function defined on $e$ such that $0 \leq B_e \leq 1$. Consequently, \eqref{eqn:R_K^0 second} implies that $R_K^0 z\!\!\mid_e = 0$ on each edge. This in turn means that $R_K^0 z = B_K \phi_p$, where $\phi_p \in \cQ_p(K)$ and $B_K$ is the biquadratic bubble function defined on $K$ such that $0 \leq B_K \leq 1$ and $B_K\!\!\mid_{\partial K} = 0$. Now \eqref{eqn:R_K^0 first} implies that $R_K^0 z = 0$. The mesh regularity hypothesis and a scaling argument guarantee the validity of \eqref{eqn:R_K^0 third} with a constant $C$ independent of $K$.
\end{proof}
We can now construct a projector into the enriched finite element space such that the $H^1$-norm is bounded by an $h$-independent number. This is the content of the following Lemma.

\begin{lemma} \label{lemma:R_K}
There exists a projector $R_K$ from $H^1(K)$ into $S_{p+2}(K)$ such that 
\begin{align}
\scp{R_K z}{v}_{K} &= \scp{z}{v}_{K} \quad \forall v \in S_p(K) \label{eqn:R_K first} \\ 
\dual{R_K z}{\gamma}_{\partial K} &= \dual{z}{\gamma}_{\partial K} \quad \forall \gamma \in \Gamma_p( \partial K) \label{eqn:R_K second} \\
\norm{R_K z}_{H^1(K)} 
&\leq C\norm{z}_{H^1(K)}
\label{eqn:R_K third}
\end{align}
for all $z \in H^1(K)$.
\end{lemma}
\begin{proof}
$R_K z$ is defined as $R_K z = R_K^0 (z-\bar{z}) + \bar{z}$, where $\bar{z}$ is the constant function
\[
\bar{z} = \frac{\int_K z\,\mathrm{d}K}{\int_K \mathrm{d}K}
\]
which, by a scaling argument and a variant of Friedrichs' inequality, satisfies
\[
\norm{z-\bar{z}}_{L_2(K)} \leq Ch_K\abs{z}_{H^1(K)}
\]
It follows from the definition of $R_K$ that $R_K z -z = R_K^0 (z-\bar{z}) - (z-\bar{z})$ so that \eqref{eqn:R_K^0 first} and \eqref{eqn:R_K^0 second} imply \eqref{eqn:R_K first} and \eqref{eqn:R_K second}.

We have
\[
\norm{R_K z}_{L_2(K)} \leq \norm{R_K^0(z-\bar{z})}_{L_2(K)} + \norm{\bar{z}}_{L_2(K)}
\]
\end{proof}

\begin{lemma} \label{lemma:pi_K}
There exists an operator $\vpi_K: \boldsymbol{H}(\mathrm{div},K) \rightarrow \bV\!_{p+2}(K)$ such that
\begin{align}
\scp{\vq-\vpi_K \vq}{\veta}_K &= 0 \quad \forall \veta \in \bS_{p}(K)
\label{eqn:pi_K first} \\
\dual{\gamma}{(\vq - \vpi_K \vq) \cdot \vn}_{\partial K} &= 0 \quad \forall \gamma \in \tilde{\Gamma}_{p+1}(\partial K) 
\label{eqn:pi_K second} \\
\norm{\vpi_K \vq}_{\boldsymbol{H}(\mathrm{div},K)} &\leq C \norm{\vq}_{\boldsymbol{H}(\mathrm{div},K)}
\label{eqn:pi_K third}
\end{align}
for all $\vq \in \boldsymbol{H}(\mathrm{div},K)$.
\end{lemma}
\begin{proof}
We start by constructing a bounded projector $\vpi_{\hat{K}} : \boldsymbol{H}(\mathrm{div},\hat{K}) \rightarrow \boldsymbol{\cQ}_{p+2}(\hat{K})$ for the rectangular master element $\hat{K}$. The construction is based on the observation that \eqref{eqn:pi_K first} and \eqref{eqn:pi_K second} resemble closely the canonical degrees of freedom in the Raviart-Thomas space $\boldsymbol{\cR\cT}_{p+1}(\hat{K}) = \cP_{p+2,p+1}(\hat{K}) \times \cP_{p+1,p+2}(\hat{K})$. Namely, if we denote by $\Gamma_{p+1}^\perp(\partial \hat{K})$ the $L^2(\partial \hat{K})$-orthogonal complement of $\tilde{\Gamma}_{p+1}(\partial \hat{K})$ in $\Gamma_{p+1}(\partial \hat{K})$, and define
\[
\boldsymbol{\cR}(\hat{K}) = \{ \hat{\vq} \in \boldsymbol{\cR\cT}_{p+1}(\hat{K}) \; : \; \dual{\hat{\gamma}}{\hat{\vq} \cdot \hat{\vn}}_{\partial K} = 0 \quad \forall \hat{\gamma} \in \Gamma_{p+1}^\perp(\partial \hat{K}) \}
\]
then the operator $\vpi_{\hat{K}}: \boldsymbol{H}(\mathrm{div},\hat{K}) \rightarrow \boldsymbol{\cR}(\hat{K})$ is indeed well-defined by the conditions
\[
\begin{aligned}
\scp{\vpi_{\hat{K}} \hat{\vq}}{\hat{\veta}}_{\hat{K}} &= \scp{\hat{\vq}}{\hat{\veta}}_{\hat{K}} \quad \forall \hat{\veta} \in \cP_{p,p+1}(\hat{K}) \times \cP_{p+1,p}(\hat{K})
\\
\dual{\vpi_{\hat{K}} \hat{\vq} \cdot \hat{\vn}}{\hat{\eta}}_{\partial \hat{K}} &= \dual{\hat{\vq} \cdot \hat{\vn}}{\hat{\eta}}_{\partial \hat{K}} \quad \forall \hat{\eta} \in \tilde{\Gamma}_{p+1}(\partial \hat{K}) 
\end{aligned}
\]
This is true because $\vpi_{\hat{K}} \hat{\vq} \in \cR(\hat{K})$ is a function in $\boldsymbol{\cR\cT}_{p+1}(\hat{K})$ and all of its degrees of freedom must vanish when $\hat{\vq} = \boldsymbol{0}$.

The corresponding projection for an arbitrary element $K = \mF_K(\hat{K})$ can be defined using the Piola transform as $\vpi_K = \mP_K \circ \vpi_{\hat{K}} \circ \mP_K^{-1}$. We have $\mJ_K^T \hat{\veta} \in \cP_{p,p+1}(\hat{K}) \times \cP_{p+1,p}(\hat{K})$ whenever $\hat{\veta} \in \cQ_p(\hat{K})$ so that \eqref{eqn:pi_K first} and \eqref{eqn:pi_K second} follow from the identities
\[
\begin{aligned}
&\scp{\vq-\vpi_K \vq}{\veta}_K = \scp{\mJ_K (\hat{\vq} - \vpi_{\hat{K}} \hat{\vq})}{\hat{\veta}}_{\hat{K}} = \scp{\hat{\vq} - \vpi_{\hat{K}} \hat{\vq}}{\mJ_K^T \hat{\veta}}_{\hat{K}} \\
&\dual{(\vq - \vpi_K \vq) \cdot \vn}{\vgamma}_{\partial K} = \dual{(\hat{\vq} - \vpi_{\hat{K}} \hat{\vq}) \cdot \hat{\vn}}{\hat{\vgamma}}_{\partial \hat{K}}
\end{aligned}
\]

To prove \eqref{eqn:pi_K third}, we first assume that $h_K=1$ and notice that $\vpi_{\hat{K}}$ from $\bH(\mathrm{div},\hat{K})$ to $\bL_2(\hat{K})$, $\mP_K$ from $\bH(\mathrm{div},\hat{K})$ to $\bH(\mathrm{div},K)$ and $\mP_K^{-1}$ from $\bH(\mathrm{div},K)$ to $\bH(\mathrm{div},\hat{K})$ are bounded operators with bounds depending only on the shape of $K$. Therefore, $\vpi_K$ is bounded from $\bH(\mathrm{div},K)$ to $\bL_2(K)$. 

To extend the $\bL_2(K)$-bound to an arbitrary convex quadrilateral $K$, we follow \cite{Arnold2005} and introduce the dilated element $\tilde{K} = \mD(K)$ defined by $\mD(\vx) = h_K^{-1} \vx$. We have then $\mF_{\tilde{K}} = \mD \circ \mF_K$ so that $\vpi_{\tilde{K}} = \mP_{\tilde{K}} \circ \vpi_K \circ \mP_{\tilde{K}}^{-1}$ and for any $\vq \in \bH(\mathrm{div},K)$, let $\tilde{\vq} = h_K \vq(h_K \tilde{\vx})$. Then,
\[
\begin{aligned}
\norm{\tilde{\vq}}_{\bL_2(\tilde{K})} &= \norm{\vq}_{\bL_2(K)} \\
\norm{\tilde{\vnabla} \cdot \tilde{\vq}}_{\bL_2(\tilde{K})} &= h_K^2 \norm{\vnabla \cdot \vq}_{\bL_2(\tilde{K})} = h_K \norm{\vnabla \cdot \vq}_{L_2(K)}
\end{aligned}
\]
so that we have
\[
\norm{\vpi_K \vq}_{\bL_2(K)} = \norm{h_K^{-1} \vpi_{\tilde{K}} \hat{\vq}}_{\bL_2(K)} = \norm{\vpi_{\hat{K}} \vq}_{\bL_2(\tilde{K})} \leq C \norm{\tilde{\vq}}_{\bH(\mathrm{div},\tilde{K})} \leq C(\norm{\vq}_{\bL_2(K)} + h_K \norm{\vnabla \cdot \vq}_{L_2(K)})
\]

To obtain an $h$-independent bound for the norm of the divergence, we use the identities
\[
\begin{aligned}
(\vnabla \cdot \vq) \circ \mF_K &= \frac{\hat{\vnabla} \cdot \hat{\vq}}{\det \mJ_K(\hat{\vx})} \\
\hat{\vnabla} \cdot (\vpi_{\hat{K}} \hat{\vq}) &= \hat{Q}_{p+1} \hat{\vnabla} \cdot \hat{\vq}
\end{aligned}
\]
where $\hat{Q}_{p+1}$ denotes the $L_2(\hat{K})$-projector onto $\cQ_{p+1}(\hat{K})$, to write
\[
\vnabla \cdot (\vpi_K \vq) = \frac{\hat{\vnabla} \cdot (\vpi_{\hat{K}} \hat{\vq})}{\det \mJ_K} = \frac{\hat{\Pi}_{p+1} \hat{\vnabla} \cdot \hat{\vq}}{\det \mJ_K} =  \frac{\hat{\Pi}_{p+1} [\det \mJ_K (\vnabla \cdot \vq) \circ \mF_K]}{\det \mJ_K}
\]
In other words $\vnabla \cdot (\vpi_K \vq) = \Lambda_K (\vnabla \cdot \vq)$, where $\Lambda_K:L_2(K) \rightarrow L_2(K)$ is defined by
\[
\Lambda_K f = \frac{\hat{\Pi}_{p+1} [\det(\mJ_K) (f \circ \mF_K)]}{\det(\mJ_K)} \circ \mF_K^{-1}
\]
for any scalar function $f$.
Now \eqref{eqn:pi_K third} follows because:
\begin{equation} \label{eqn:bound for Lambda_K}
\norm{\Lambda_K f}_{L_2(K)} \leq C \norm{f}_{L_2(K)} \quad \forall f \in L_2(K)
\end{equation}
The bound \eqref{eqn:bound for Lambda_K} is obvious for elements with unit diameter and can be extended to elements with arbitrary diameter with a constant depending only on the shape of $K$ by using the dilation $\vx \hookrightarrow h_K^{-1} \vx$.
\end{proof}

We can now state our main approximation result:
\begin{theorem}
Let $\vu = (\mV,\mM,w,\vpsi,r_h,\hat{w},\hat{\vpsi},\hat{V}_n,\hat{\mM}_n)$ denote the exact solution to the Reissner-Mindlin model and $\vu_h = (\mV_h,\mM_h,w_h,\vpsi_h,r_h,\hat{w}_h,\hat{\vpsi}_h,\hat{V}_{n,h},\hat{\mM}_{n,h})$ the DPG approximation of degree $p$ on an affine mesh with maximum element diameter $h$. The approximation error
\[
\begin{split}
e &= \norm{\mV-\mV_h}_{L_2(\Omega)} + \norm{\mM-\mM_h}_{L_2(\Omega)} + \norm{w-w_h}_{L_2(\Omega)} + \norm{\vpsi-\vpsi_h}_{L_2(\Omega)} + \norm{r-r_h}_{L_2(\Omega)} \\
&+ \norm{\hat{w}-\hat{w}_h}_{H^{1/2}(\partial \Omega_h)} + \norm{\hat{\vpsi}-\hat{\vpsi}_h}_{\boldsymbol{H}^{1/2}(\partial \Omega_h)} + \norm{\hat{V}_n-\hat{V}_{n,h}}_{H^{-1/2}(\partial \Omega_h)} + \norm{\hat{\mM}_n-\hat{\mM}_{n,h}}_{H^{-1/2}(\partial \Omega_h)}
\end{split}
\]
satisfies an a priori estimate
\begin{equation} \label{eqn:error estimate}
e \leq Ct^{-1}\left(\norm{\mV}_{\bH^{p+2}(\Omega)} + \norm{\mM}_{\bH^{p+2}(\Omega)} + \norm{w}_{H^{p+2}(\Omega)} + \norm{\vpsi}_{\bH^{p+2}(\Omega)}\right)
\end{equation}
where the constant $C$ is independent of $h$ and $t$ but depends on $p$ and $\Omega$.
\end{theorem}
\begin{proof}
We start by defining a global projection operator $\mPi_h: \cbV \rightarrow \cbV^{p+3}$ piecewise\footnote{The operator $\vpi_K$ acts on tensors row-wise.}:
\[
(\mPi_h \vv)\!\!\mid_K = (\vpi_K \mtau,\vpi_K \vq,R_K z, \boldsymbol{R}_K \vphi,Q_K \mu)
\]
where $\vpi_K$ and $R_K$ are the projectors defined in Lemmas \ref{lemma:pi_K} and \ref{lemma:R_K} and $Q_K$ is the $L_2$-projector onto $S_{p+2}(K)$. The projectors satisfy
\begin{equation} \label{eqn:orthogonality}
\begin{alignedat}{3}
\scp{k^{-1}t^2 \veta_h + \vtheta_h}{\vq - \vpi_K \vq}_K &= 0, & 
\quad \scp{v_h}{\vnabla \cdot (\vq - \vpi_K \vq)}_K &= 0, & 
\quad \dual{\hat{v}_h}{(\vq - \vpi_K \vq) \cdot \vn}_{1/2,\partial K} &= 0, \\
\scp{\cC^{-1}\msigma_h + \rho_h\mJ}{(\mtau - \vpi_K \mtau)}_K &= 0, & \scp{\theta_h}{\vnabla \cdot (\mtau - \vpi_K \mtau)}_K &= 0, & 
\quad \dual{\hat{\vtheta}_h}{(\mtau-\vpi_K\mtau)\vn}_{1/2,\partial K} &=0, \\
&& \scp{\veta_h}{\vnabla(z - R_K z)}_K &= 0, 
\quad &\dual{z-R_K z}{\hat{\eta}_{n,h}}_{\partial K} &= 0, \\
&& \scp{\sigma_h}{\vnabla(\vphi - \bR_K \vphi)}_K &= 0, &
\quad \dual{\vphi- \bR_K \vphi}{\hat{\msigma}_n}_{1/2,\partial K} &= 0, \\
\scp{\msigma_h}{(s-Q_K s)\mJ}_K &= 0 
\end{alignedat}
\end{equation}
for all $\vw_h = (\veta_h,\msigma_h,v_h,\vtheta_h,\rho_h,\hat{v}_h,\hat{\vtheta}_h,\hat{\eta}_{n,h},\hat{\msigma}_{n,h}) \in \cbU_h$. The first and third columns follow directly from Lemmas \ref{lemma:pi_K}, \ref{lemma:R_K} and the definition of $Q_K$. The second column is proved using the same Lemmas in conjunction with integration by parts. The first equality in the second column holds because
\[
\begin{aligned}
\scp{v_h}{\vnabla \cdot (\vq - \vpi_K \vq)}_K
&= \scp{\hat{v}_h}{\hat{\vnabla} \cdot (\hat{\vq} - \vpi_{\hat{K}} \hat{\vq})}_{\hat{K}} \\
&= \dual{\hat{v}_h}{(\hat{\vq}-\vpi_{\hat{K}} \hat{\vq}) \cdot \hat{\vn}}_{\partial \hat{K}} 
- \scp{\hat{\vnabla}\hat{v}_h}{\hat{\vq} - \vpi_{\hat{K}}\hat{\vq}}_K \\
&= 0
\end{aligned}
\]
The third equality in the second column holds because
\[
\begin{aligned}
\scp{\veta_h}{\vnabla(z - R_K z)}_K 
&= \dual{\veta_h \cdot \vn}{z - R_{K} z}_{\partial K} - \scp{\vnabla \cdot \veta_h}{z-R_{K} z}_{K} \\
&= \dual{\hat{\veta}_h \cdot \hat{\vn}}{\hat{z} - R_{\hat{K}} \hat{z}}_{\partial \hat{K}} - \scp{\hat{\vnabla} \cdot \hat{\veta}_h}{\hat{z}-R_{\hat{K}} \hat{z}}_{\hat{K}} \\
&= 0
\end{aligned}
\]
The second and fourth equality can be proven in the same way so that we have established the condition \eqref{eqn:projection condition} of Lemma \ref{lemma:fortin}. In other words, we have established the best approximation property \eqref{eqn:best approximation property}.

A more quantitative error estimate is obtained by using results from approximation theory. For smooth enough vector and scalar fields $\mV$ and $w$, there exist interpolants $\tilde{\mV} \in \bH(\mathrm{div},\Omega_h)$ and $\tilde{w} \in H^1(\Omega_h)$ such that
\[
\tilde{\mV}\!\!\mid_K \in \mV_p(K), \quad \tilde{w}\!\!\mid_K \in S_p(K) \quad \forall K \in \Omega_h,
\]
and
\[
\begin{aligned}
\norm{\mV - \tilde{\mV}}_{\bL_2(\Omega)} &\leq C h^{p+1}\norm{\mV}_{H^{p+1}(\Omega)}, \\
\norm{w - \tilde{w}}_{L_2(\Omega)} &\leq C h^{p+1}\norm{w}_{H^{p+1}(\Omega)}.
\end{aligned}
\]
Here $\tilde{w}$ is the standard interpolant of $w$ and $\tilde{\mV}$ denotes the projection of $\mV$ to the Raviart-Thomas space, see \cite{Girault1986a,Arnold2005}.

We can also construct interpolants satisfying $\breve{w}\!\!\mid_{\partial \Omega_h} \in H^{1/2}_0(\partial \Omega_h)$ and $\breve{\mV}\cdot \vn\!\!\mid_{\partial \Omega_h} \in H^{-1/2}(\partial \Omega_h)$ such that
\[
\breve{w}\!\!\mid_{\partial K} \in \tilde{\Gamma}_{p+1}(\partial K), \quad \breve{\mV}\cdot \vn \!\!\mid_{\partial K} \in \Gamma_p(\partial K) \quad \forall K \in \Omega_h.
\]
Since the traces $\hat{w}$ and $\hat{V}_n$ associated to the exact solution equal the traces of the corresponding field variables, we are allowed to write
\[
\begin{aligned}
\min_{\hat{v}_h} \norm{\hat{w} - \hat{v}_h}_{H^{1/2}(\partial \Omega_h)} &\leq 
\norm{w-\breve{w}}_{H^{1}(\Omega)} \\
\min_{\hat{\eta}_{n,h}} \norm{\hat{V}_n - \hat{\eta}_{n,h}}_{H^{-1/2}(\partial \Omega_h)} &\leq \norm{\mV-\breve{\mV}}_{\bH(\mathrm{div},\Omega)}
\end{aligned}
\]
Defining $\breve{w}$ as the regular interpolant of $w$ with quadrilateral elements of degree $p+1$ and $\breve{\mV}$ as the projection of $\mV$ into the Arnold-Boffi-Falk space of index $p$, we obtain the error estimates
\[
\begin{aligned}
\norm{w-\breve{w}}_{H^{1}(\Omega)} &\leq Ch^{p+1}\norm{w}_{H^{p+2}(\Omega)}, \\
\norm{\mV-\breve{\mV}}_{\bH(\mathrm{div},\Omega)} &\leq Ch^{p+1}(\norm{\mV}_{\bH^{p+1}} + \norm{\vnabla \cdot \mV}_{\bH^{p+1}(\Omega)}).
\end{aligned}
\]

Identical constructions can be carried out for the remaining solution components $\vpsi$, $\mM$, $r$, $\hat{\vpsi}$, and $\hat{\mM}_n$. Hence, the estimate \eqref{eqn:error estimate} is established.
\end{proof}

\begin{remark}
Notice that the restriction of the proof to affine mesh sequences arises from the terms involving $\veta_h$ and $\msigma_h$ in the first column of \eqref{eqn:orthogonality}. Namely, when the mapping $\mF_K$ is not affine, the use of Piola transform introduces a non-constant factor $1/\det{\mJ_K}$ violating the orthogonality conditions established in Lemma \ref{lemma:pi_K}. On an affine mesh, the same terms dictate the enrichment degree to be three, since we need to apply Lemma \ref{lemma:pi_K} also when $\veta \in \bV_p(K)$. On the other hand, the use of Piola transformation for the shear force and bending moment is necessary in general to match the normals in $\hat{V}_n$ and $\hat{\mM}_n$ with the ones in $\mV \cdot \vn$ and $\mM\vn$.
\end{remark}

\begin{remark}
When bounding the approximation error of $\hat{V}_n$ and $\hat{\mM}_n$, use of Raviart-Thomas projector would imply loss of one power of $h$ in the convergence rate on a general mesh, see \cite{Arnold2005}. In the DPG approximation the resultant tractions can be extended as well to the mentioned Arnold-Boffi-Falk space defined on the reference element as $\cbA\cbB\cbF_{p}(\hat{K})=\cP_{p+2,p}(\hat{K}) \times \cP_{p,p+2}(\hat{K})$ since the normal components of the elements of this space are also polynomials of degree $p$ on the edges.
\end{remark}

\section{Numerical Results} \label{sec:numerical results}
We study the convergence of the DPG method when applied to solve the model problem proposed in \cite{Chinosi1995}. The problem consists of a fully clamped, homogeneous and isotropic square plate loaded by the pressure distribution
\[
\begin{split}
p(x,y) = \frac{1}{12(1 - \nu^2)} [12 y (y - 1) (5 x^2 - 5 x + 1) (2 y^2 (y - 1)^2 + 
      x (x - 1) (5 y^2 - 5 y + 1)) \\
      + 
   12 x (x - 1) (5 y^2 - 5 y + 1) (2 x^2 (x - 1)^2 + 
      y (y - 1) (5 x^2 - 5 x + 1))]
\end{split}
\]
on the computational domain $\Omega = (0,1) \times (0,1)$. The problem has a closed form analytic solution than can be used to address the accuracy of numerical solution schemes. 

We use the values $\nu = 0.3$ and $\kappa = 5/6$ for the Poisson ratio and the shear correction factor, respectively. We set $p=1$ and compute the DPG solution using uniform and trapezoidal $N \times N$-meshes, with $N$ varying as $N=4,8,16,32,64$, see Fig.~\ref{fig:meshes}. 
\begin{figure}[h]
\centering
\mbox{
\includegraphics[width=0.15\linewidth]{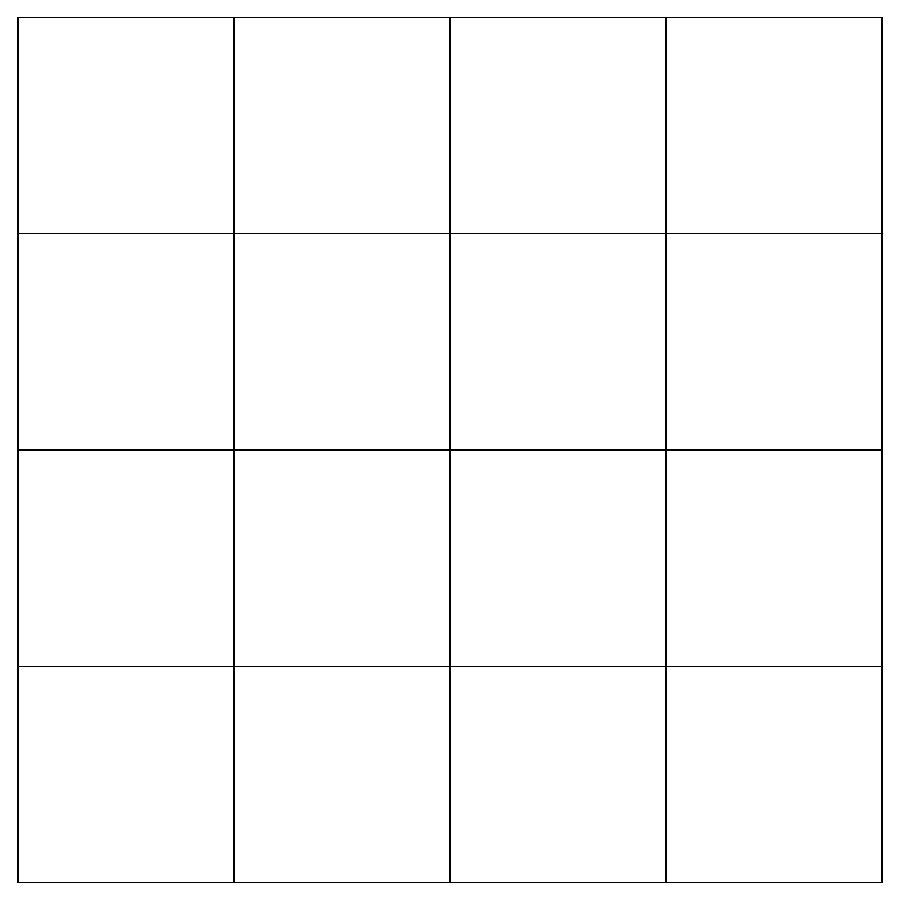}
~\quad~\includegraphics[width=0.15\linewidth]{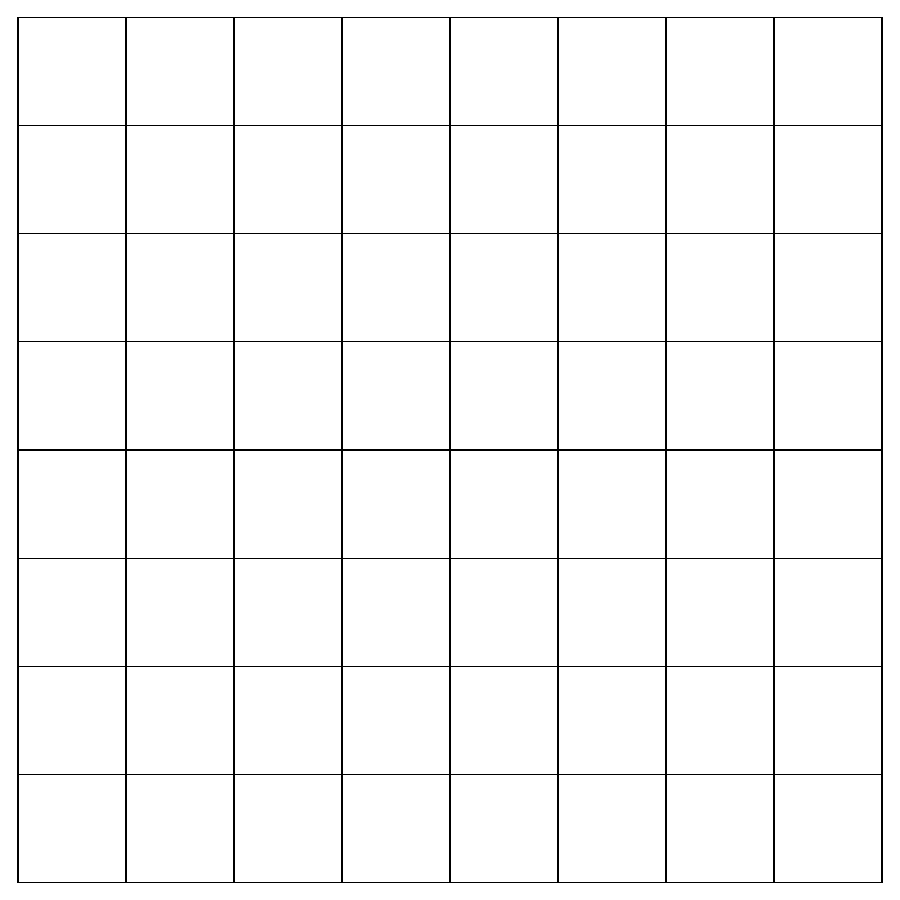}
~\quad~\includegraphics[width=0.15\linewidth]{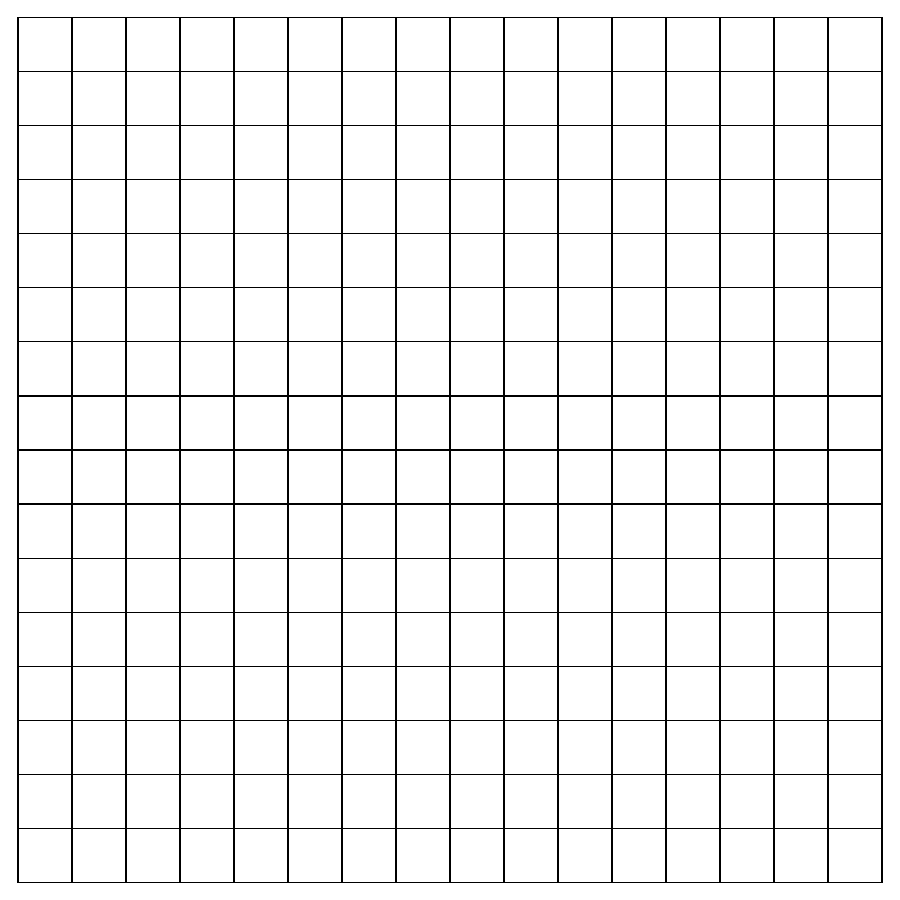}} \quad 
\parbox{0.2\linewidth}{(Uniform) \\ \vspace{20mm}} 
\mbox{
\includegraphics[width=0.15\linewidth]{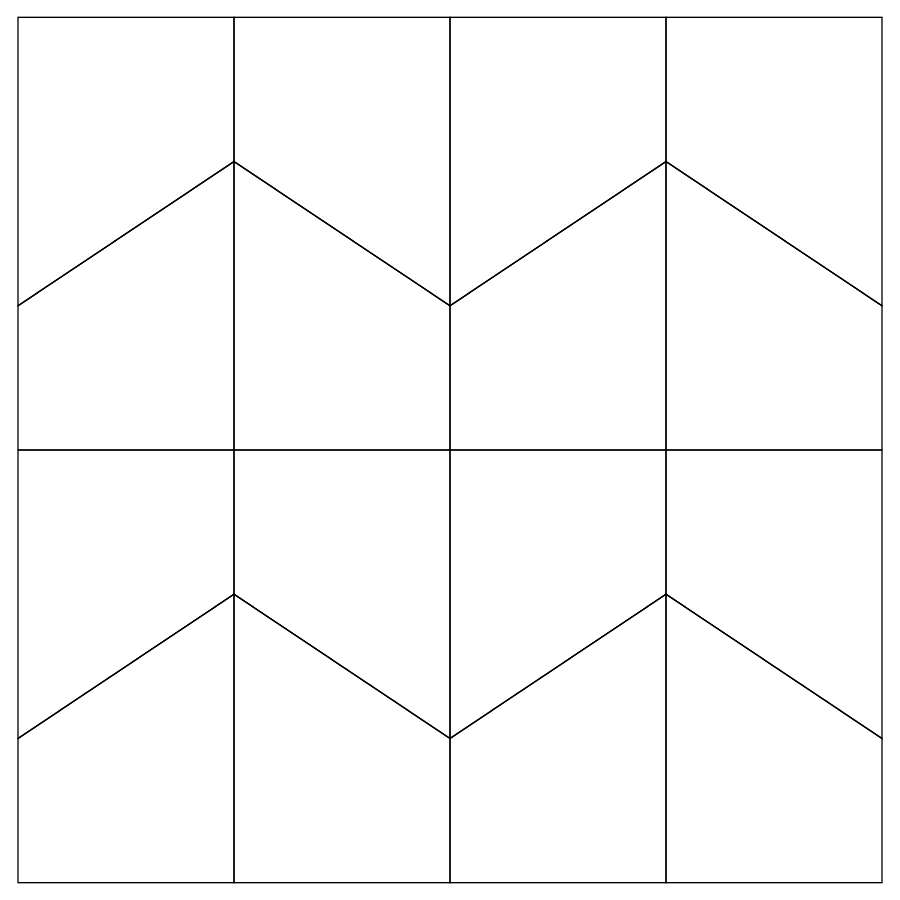}
~\quad~\includegraphics[width=0.15\linewidth]{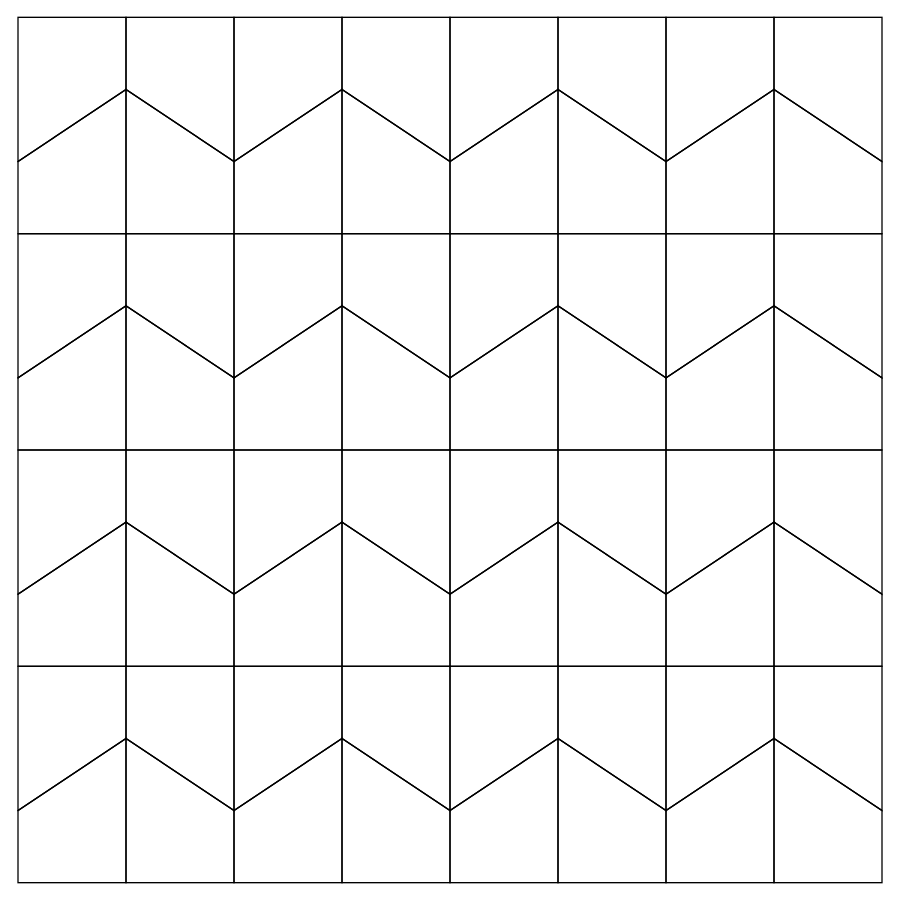}
~\quad~\includegraphics[width=0.15\linewidth]{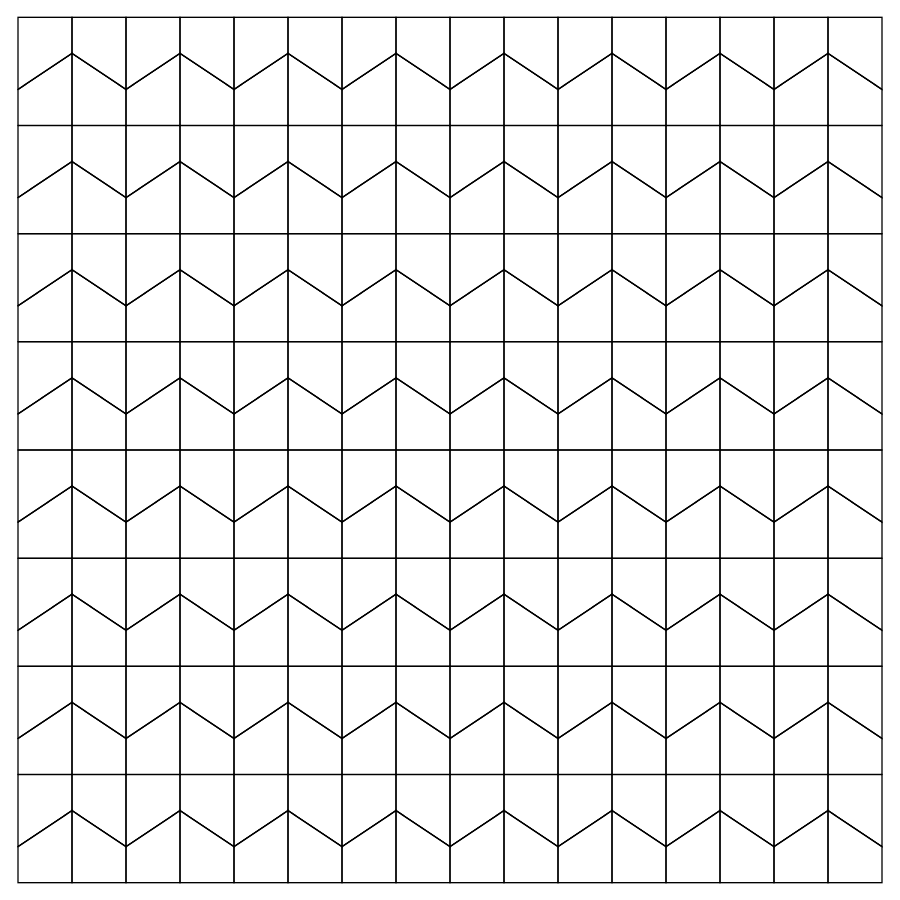}} \quad 
\parbox{0.2\linewidth}{(Trapezoidal) \\ \vspace{20mm}}
\caption{Uniform and trapezoidal mesh sequences.}
\label{fig:meshes}
\end{figure}

The results for the thickness values $t=0.1$ and $t=0.001$ are summarized in Figs.~\ref{fig:PlateErrors_10} and \ref{fig:PlateErrors_1000}, respectively. In the figures we show the relative errors in the $L_2$ norm for all quantities of interest:
\[
\frac{\norm{\mV-\mV_h}}{\norm{\mV}}, \quad \frac{\norm{\mM-\mM_h}}{\norm{\mM}}, \quad \frac{\norm{w-w_h}}{\norm{w}}, \quad \frac{\norm{\vpsi-\vpsi_h}}{\norm{\vpsi}}
\]
\begin{figure}[h]
\centering
\parbox{0.5\linewidth}{\centering Uniform \\ \includegraphics[width=\linewidth]{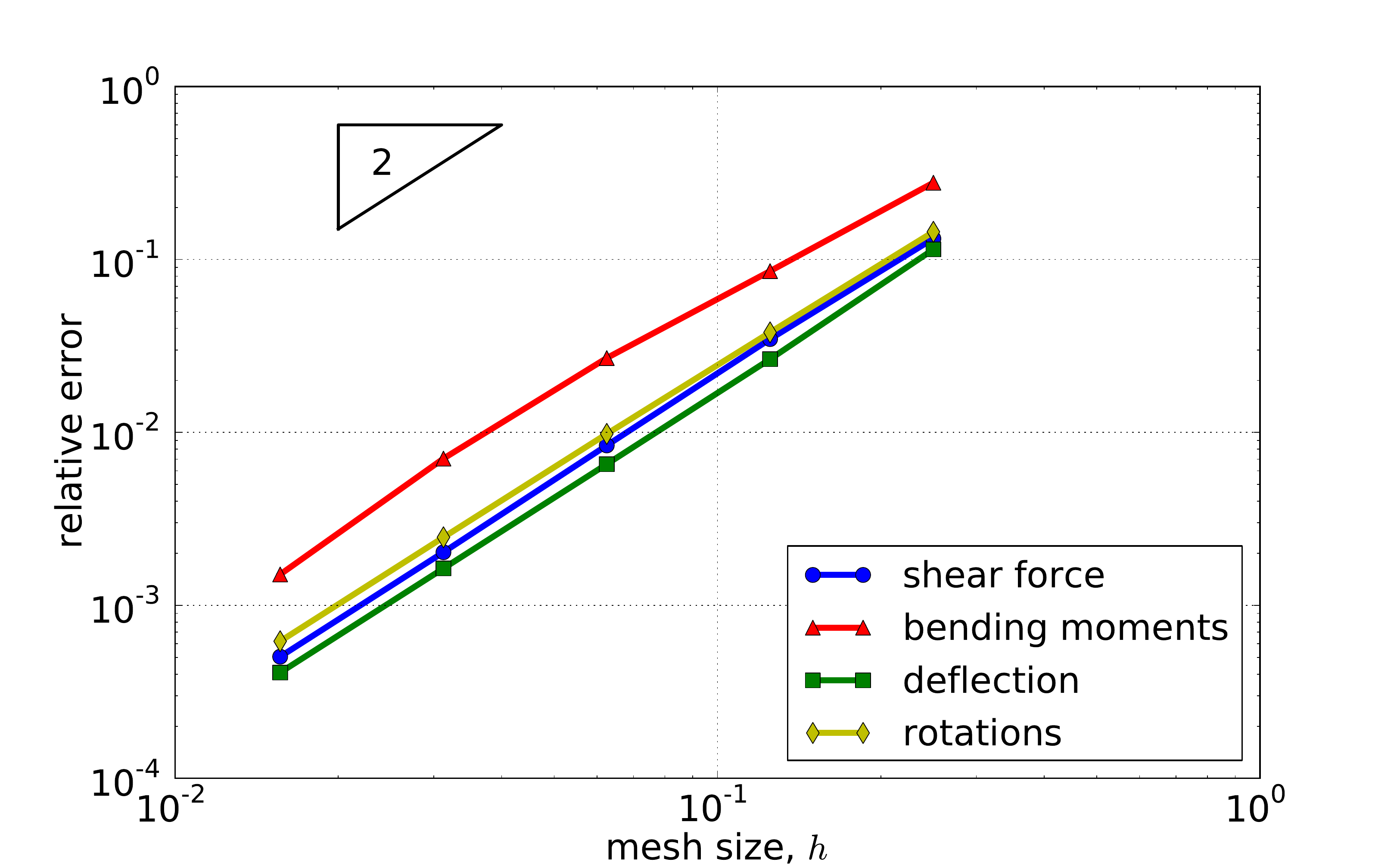}}~
\parbox{0.5\linewidth}{\centering Trapezoidal \\ \includegraphics[width=\linewidth]{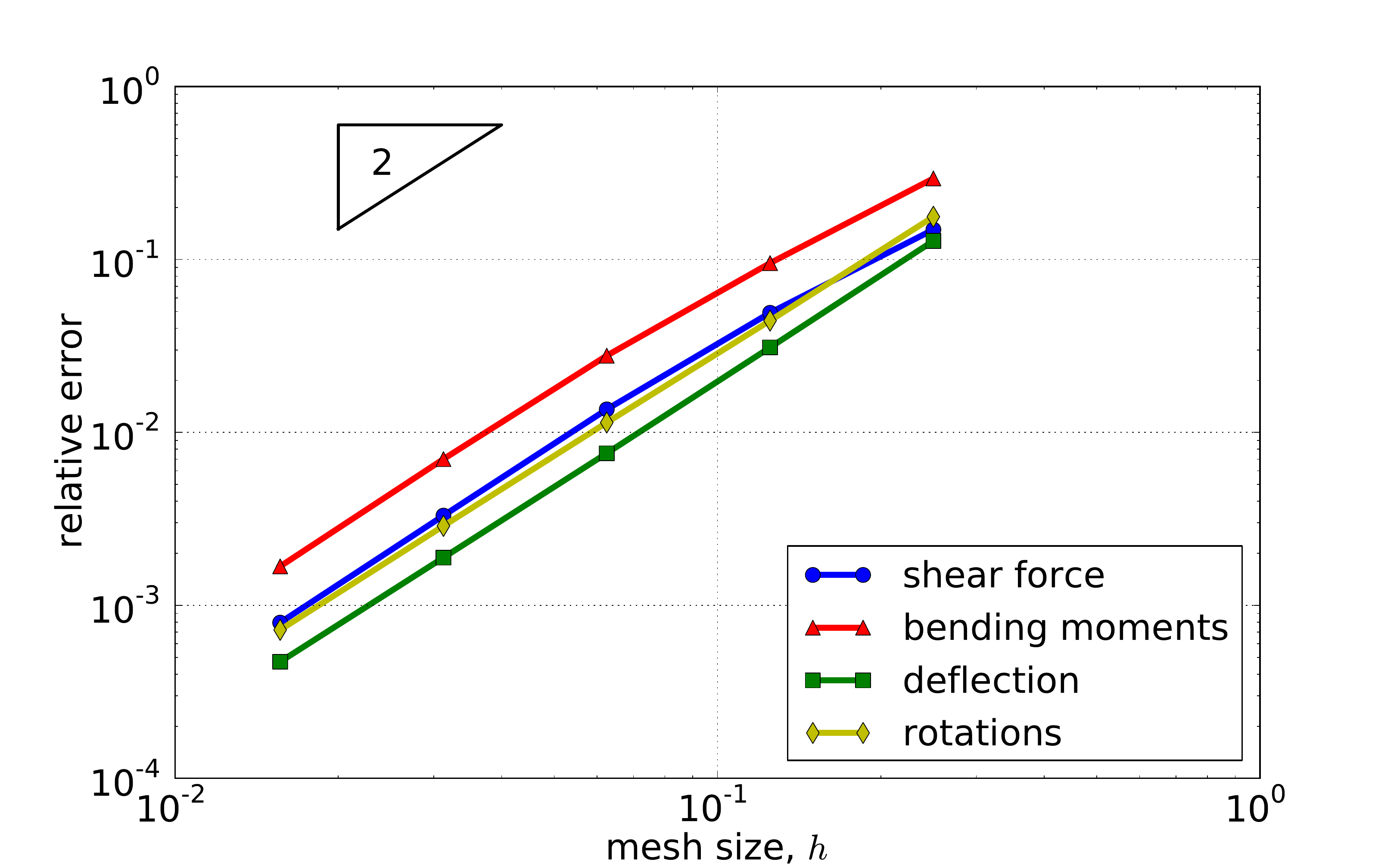}}\caption{Convergence at $t=1/10$. Uniform versus trapezoidal mesh sequence.}
\label{fig:PlateErrors_10}
\end{figure}
\begin{figure}[h]
\centering
\parbox{0.5\linewidth}{\centering Uniform \\ \includegraphics[width=\linewidth]{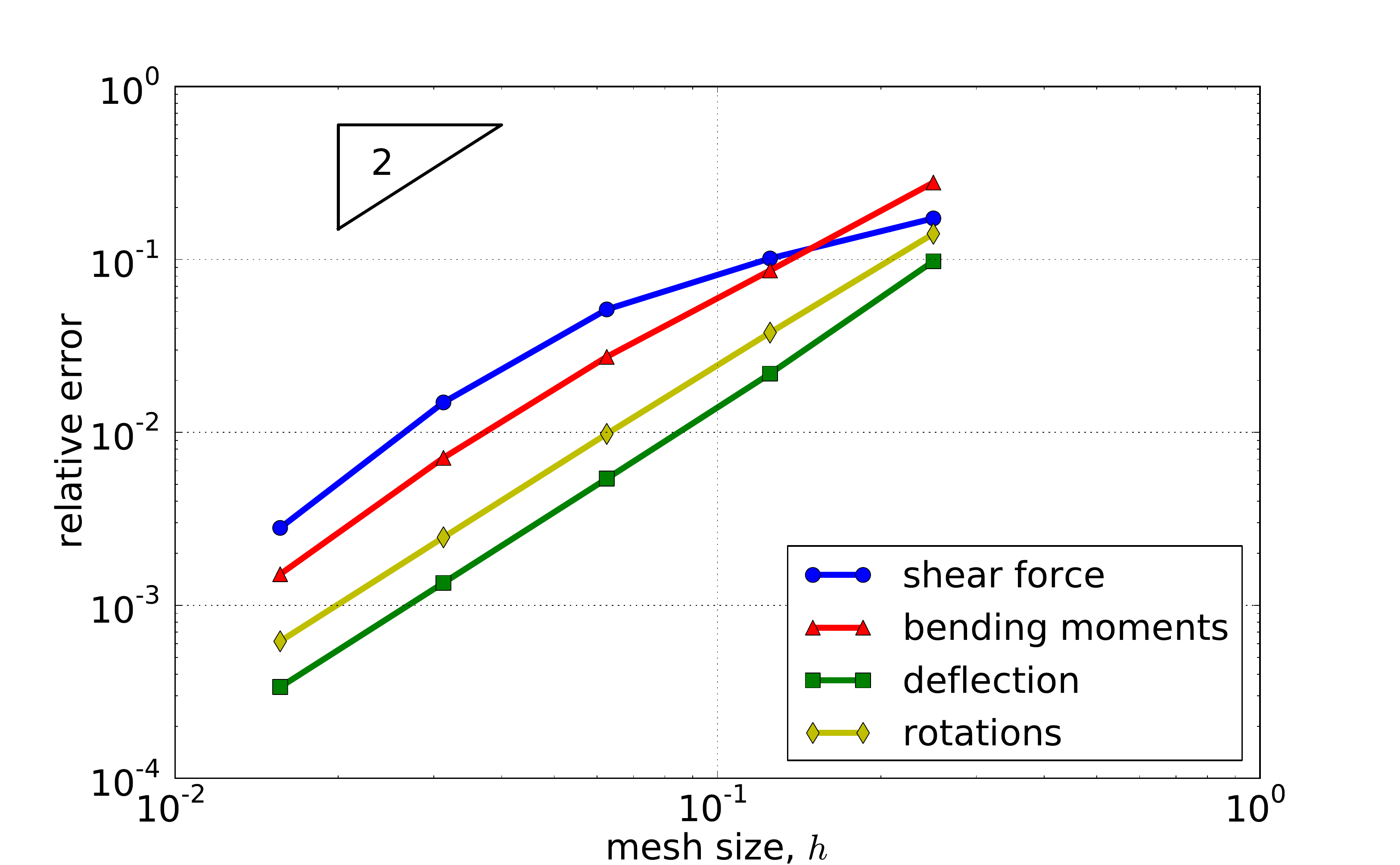}}~
\parbox{0.5\linewidth}{\centering Trapezoidal \\ \includegraphics[width=\linewidth]{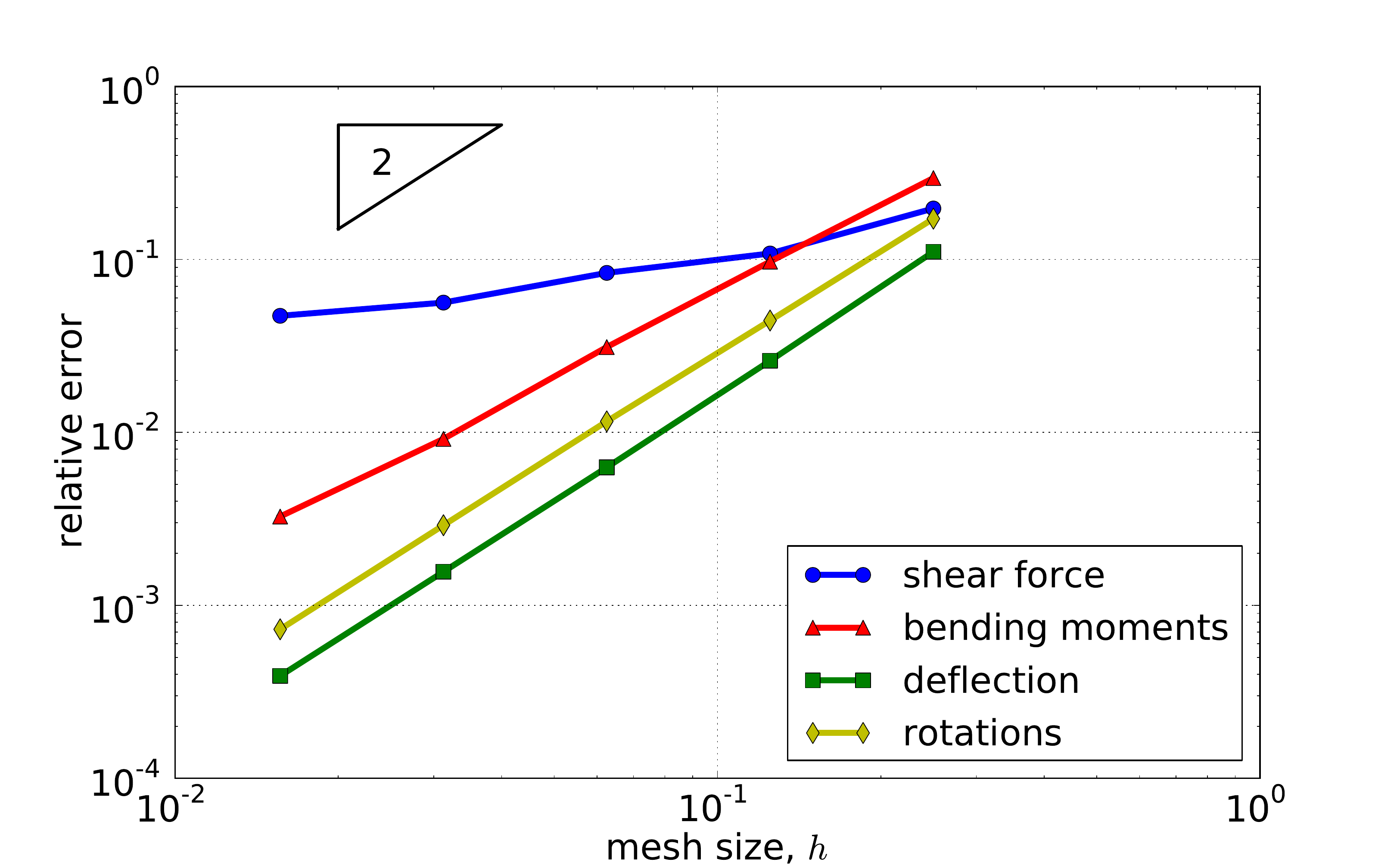}}\caption{Convergence at $t=1/1000$. Uniform versus trapezoidal mesh sequence.}
\label{fig:PlateErrors_1000}
\end{figure}
The results show that 
\begin{enumerate}
\item Optimal quadratic convergence is attained for all quantities on both mesh sequences at $t=1/10$. 
\item Convergence of the shear stress slows down at $t=1/1000$ especially on the trapezoidal mesh sequence. However, a relative error of less than 10 percent is attained also at the $16 \times 16$ trapezoidal mesh.
\end{enumerate} 

Finally, we show in Figs.~\ref{fig:shear forces}--\ref{fig:deflection} contour plots of all quantities of interest at $t=1/1000$ obtained with DPG by using a fine mesh. The good approximation quality of all quantities makes prediction of the values and the locations of maximum stresses straightforward.

\begin{figure}
\centering
\includegraphics[width = 0.33\textwidth]{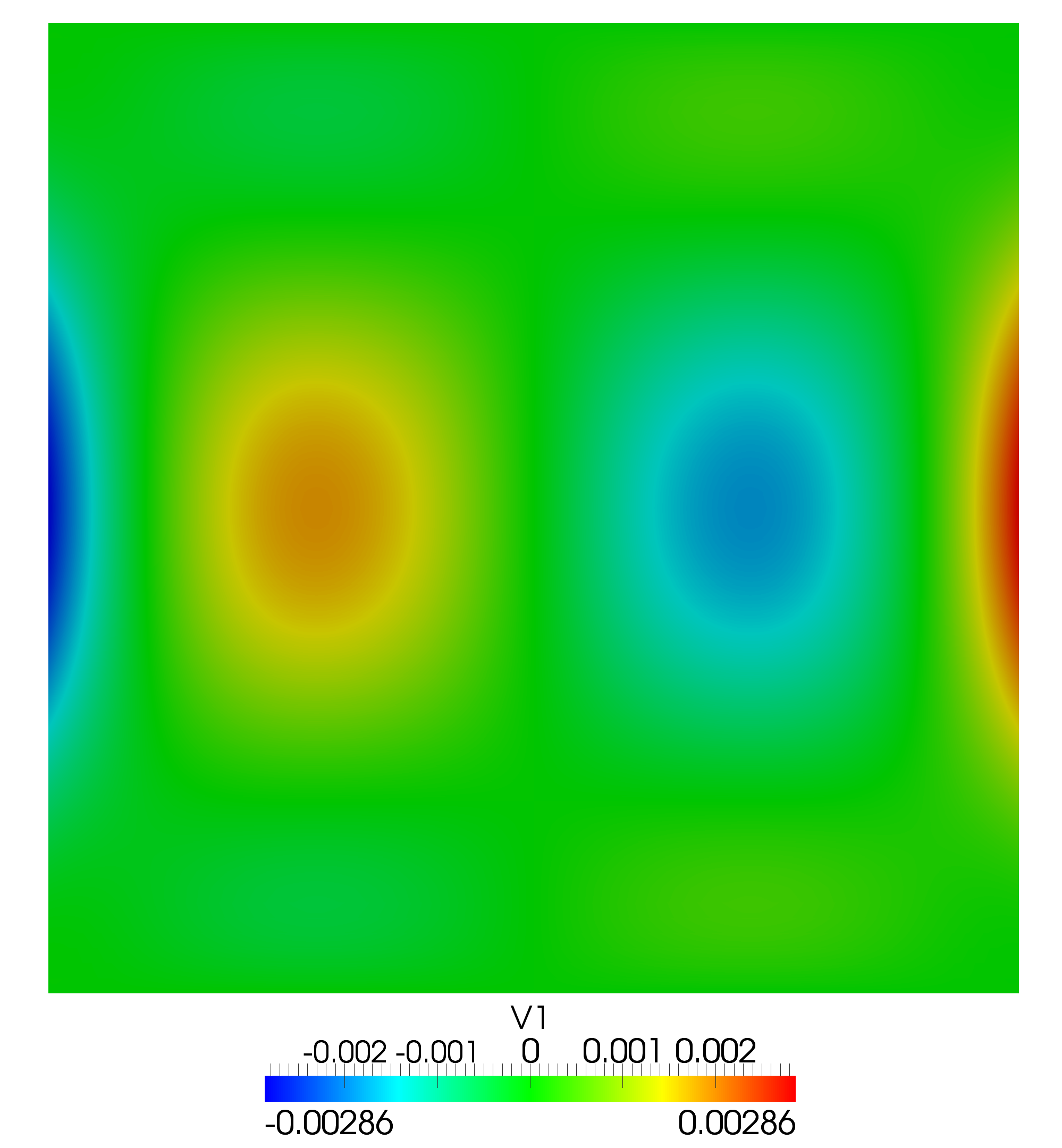}~\includegraphics[width = 0.33\textwidth]{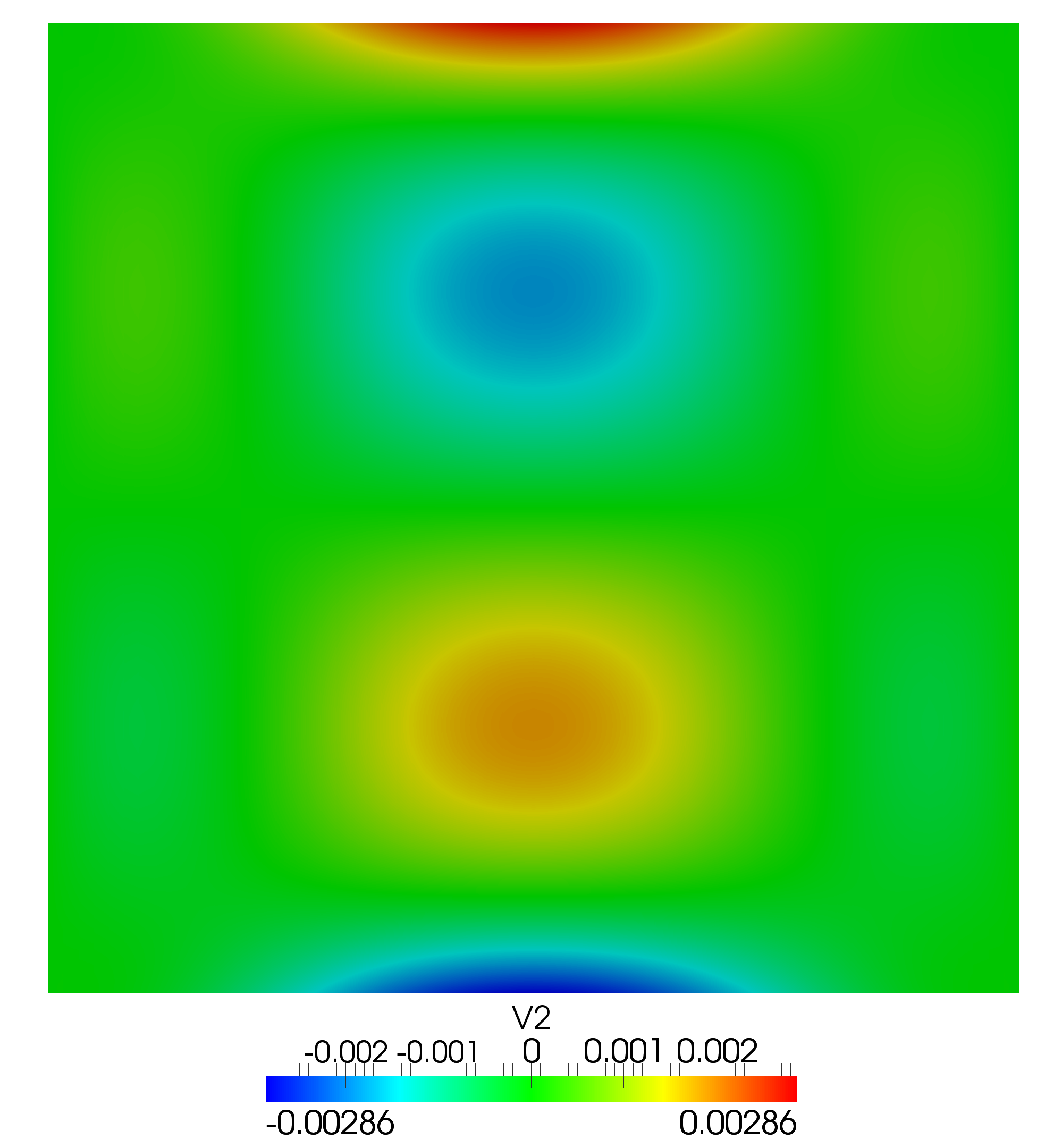}
\caption{Shear forces at $t=1/1000$.}
\label{fig:shear forces}
\end{figure}
\begin{figure}
\centering
\includegraphics[width = 0.33\textwidth]{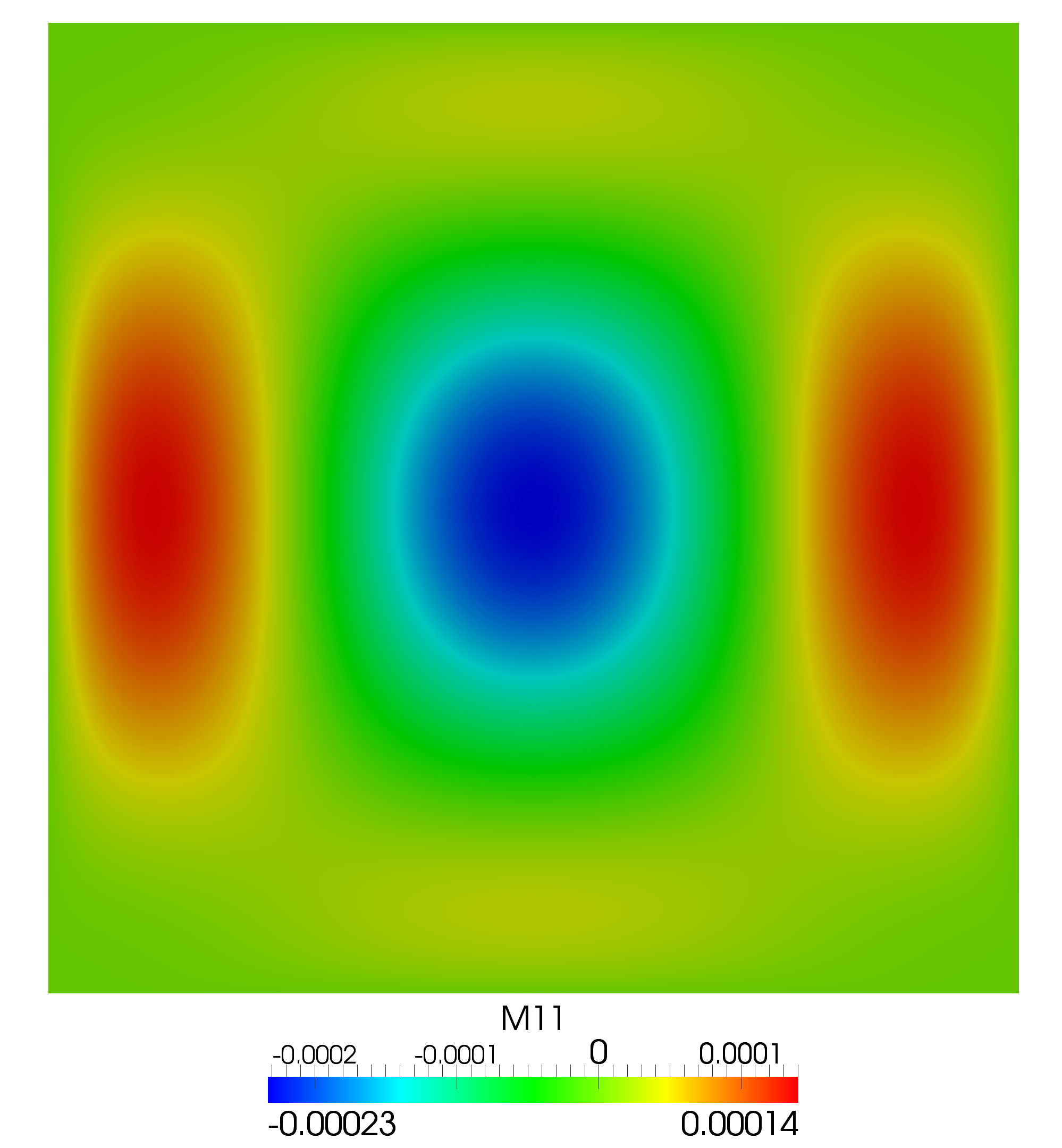}~
\includegraphics[width = 0.33\textwidth]{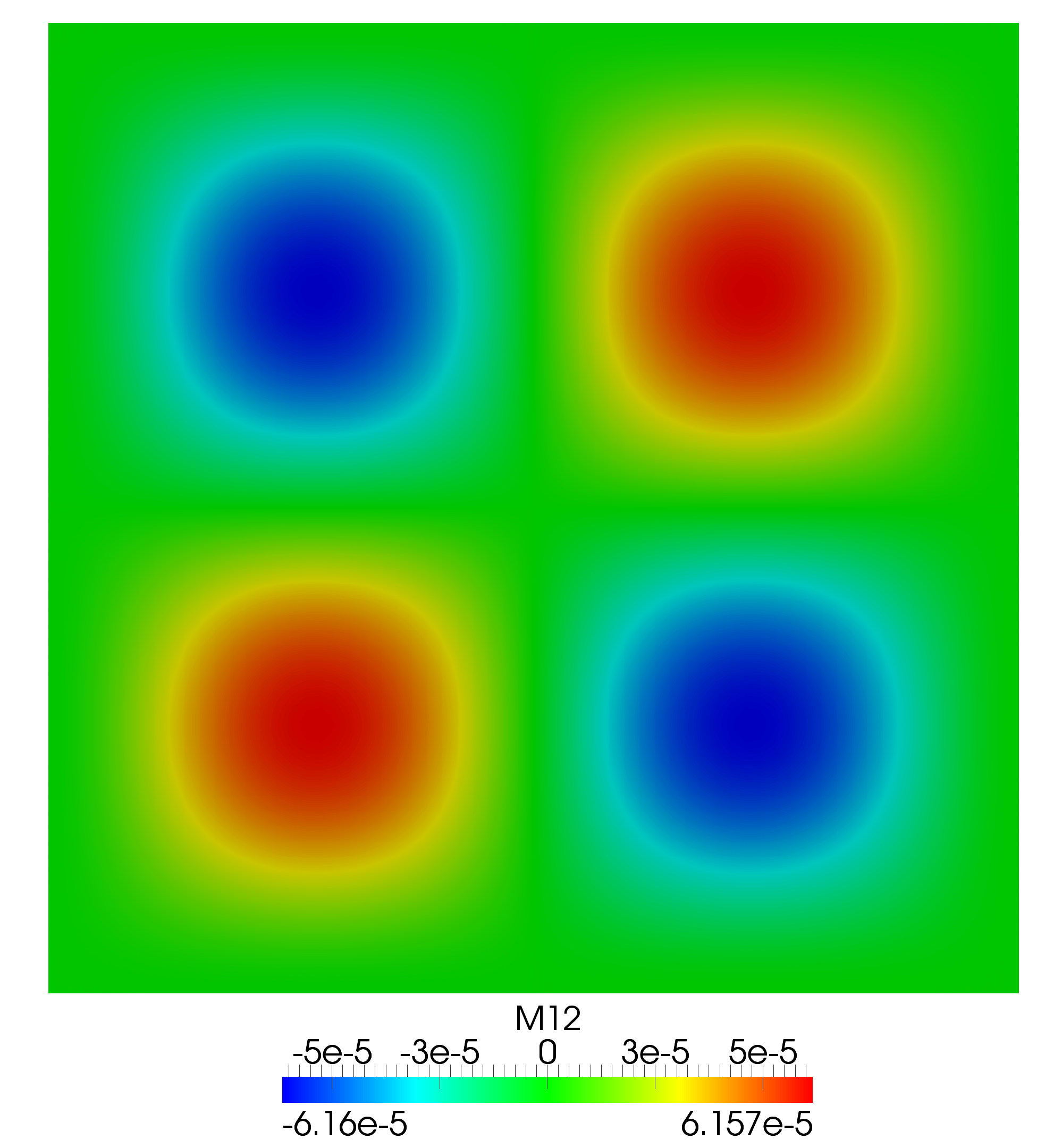}\\
\includegraphics[width = 0.33\textwidth]{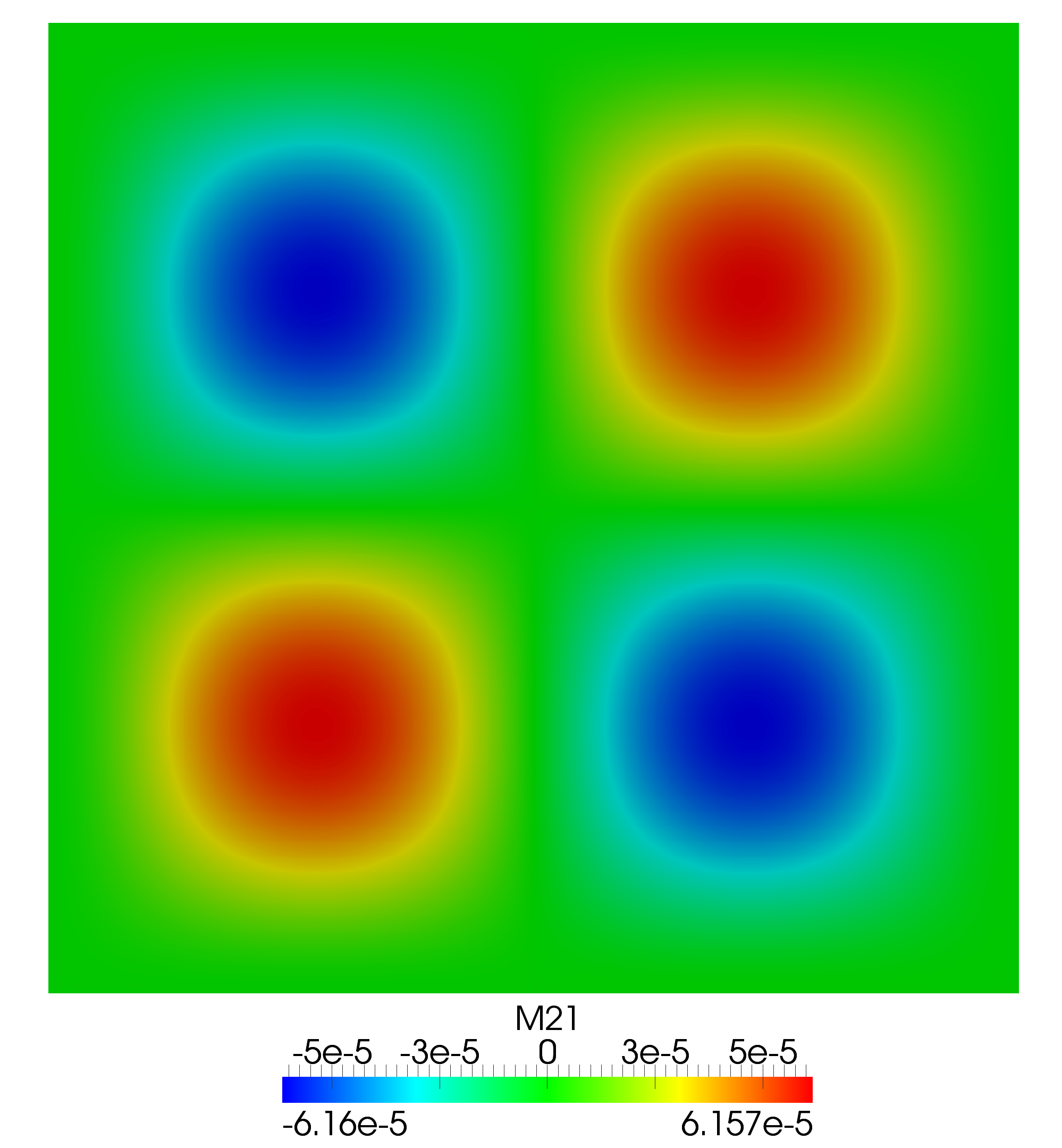}~
\includegraphics[width = 0.33\textwidth]{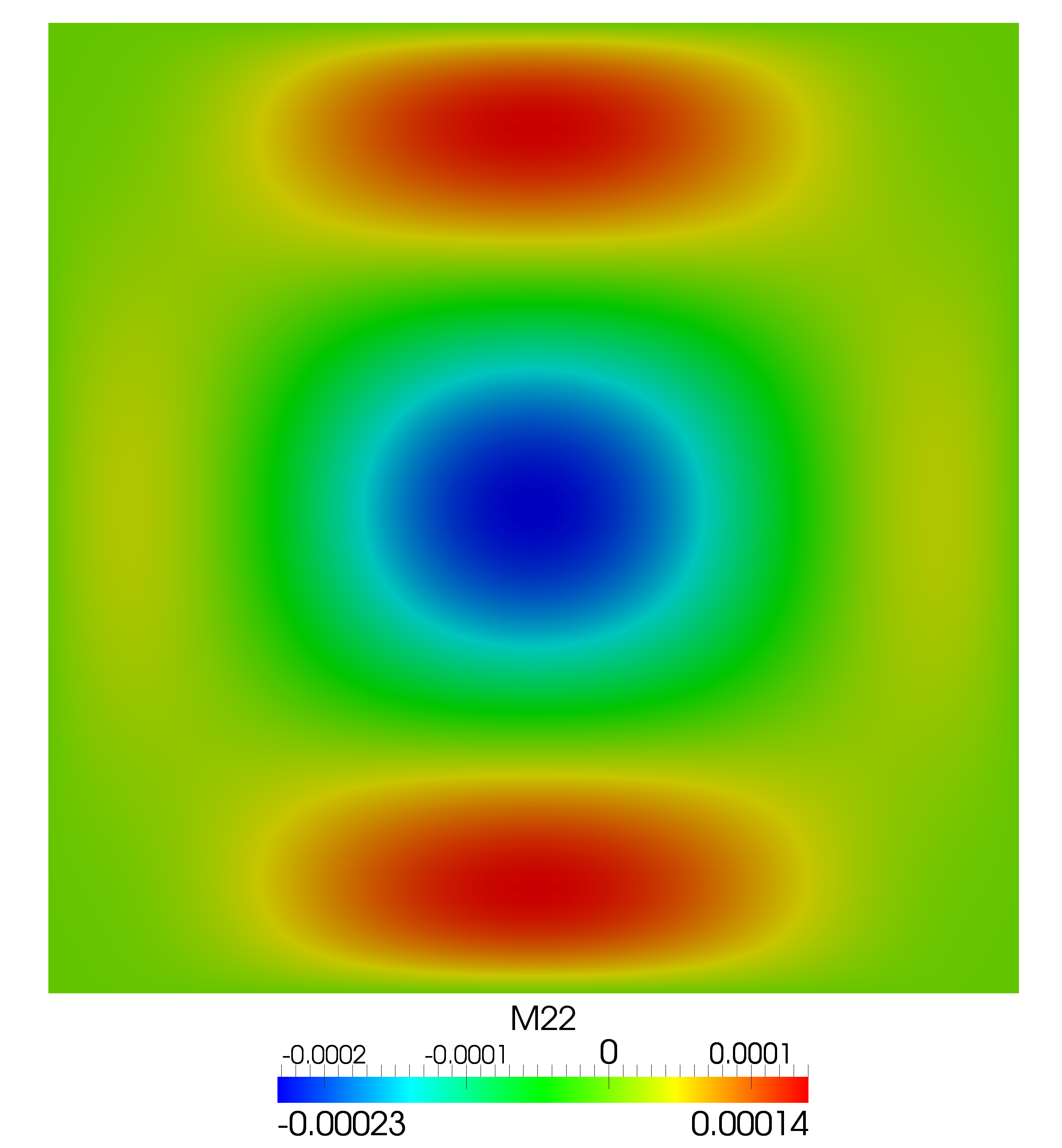}
\caption{Bending moments at $t=1/1000$.}
\end{figure}
\begin{figure}
\centering
\includegraphics[width = 0.33\textwidth]{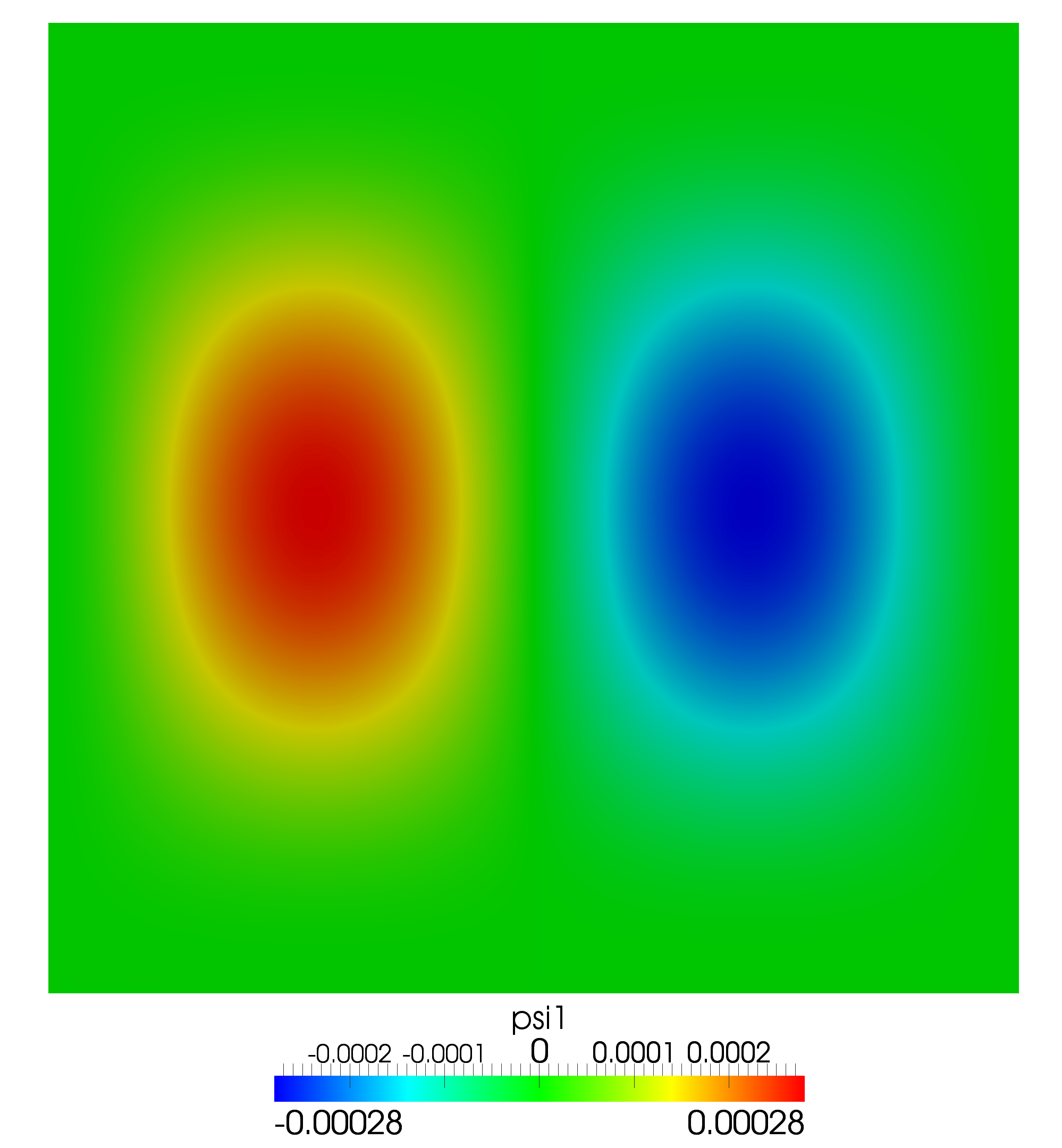}~
\includegraphics[width = 0.33\textwidth]{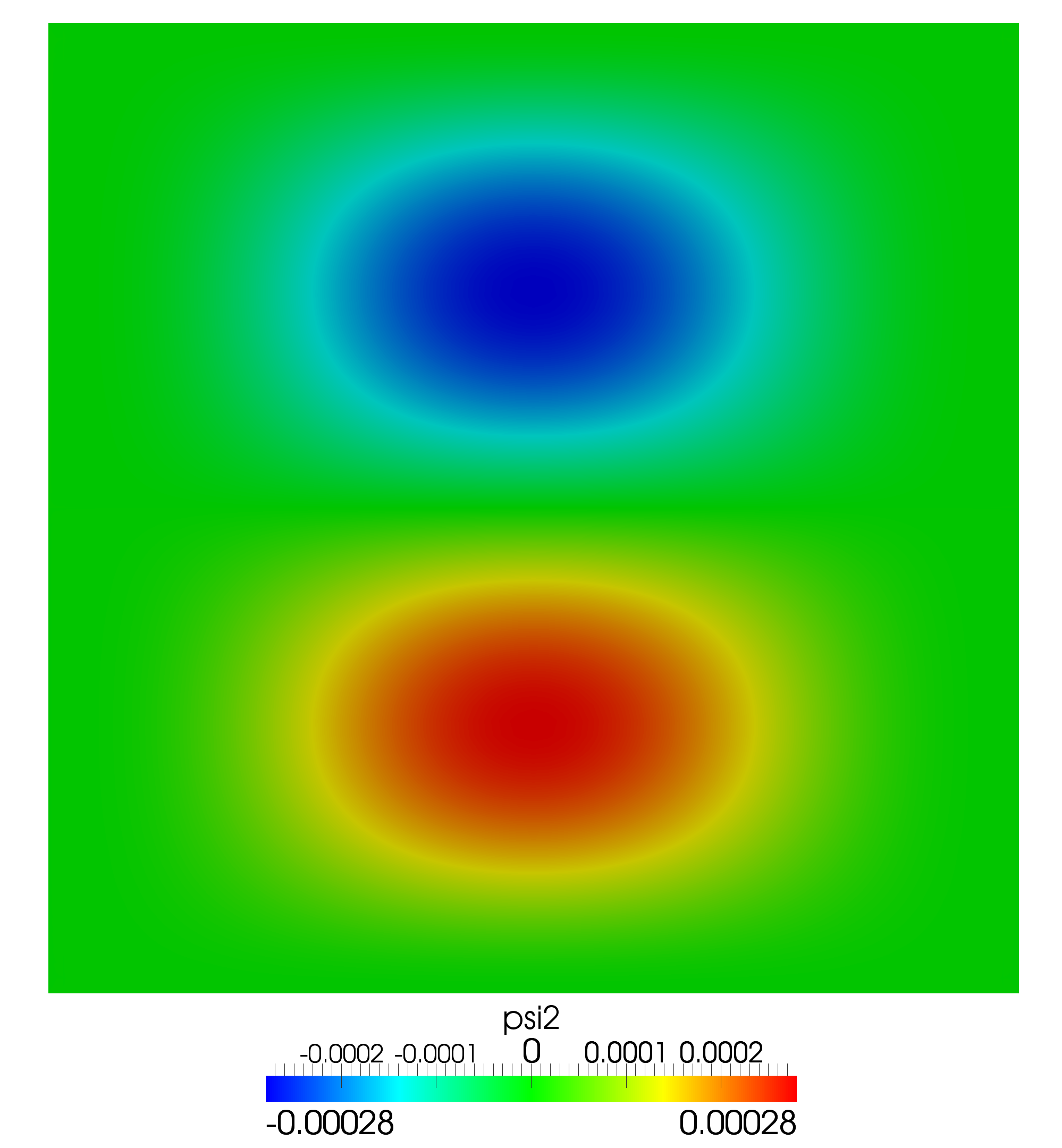}
\caption{Rotations at $t=1/1000$.}
\end{figure}
\begin{figure}
\centering
\includegraphics[width = 0.33\textwidth]{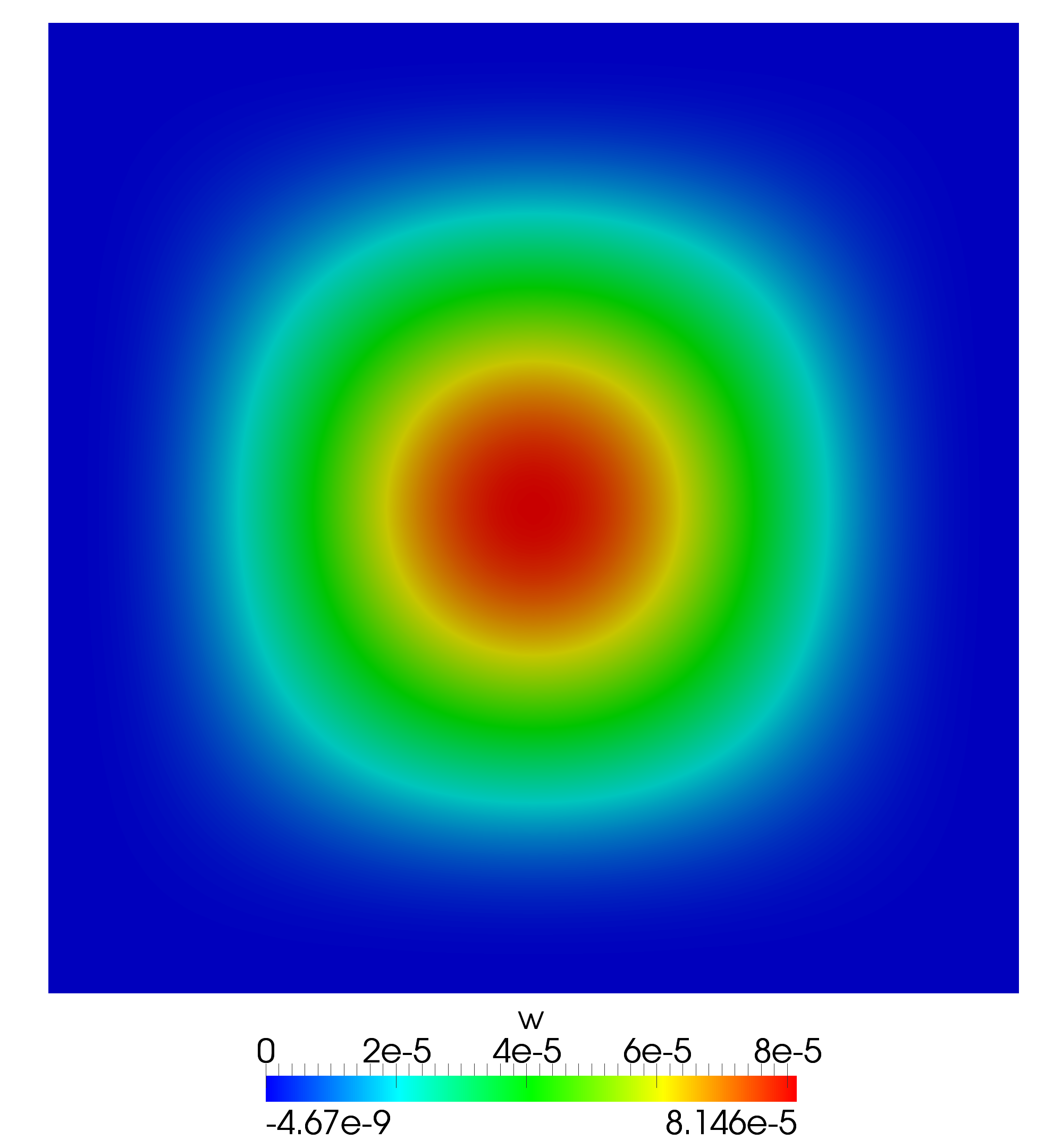}
\caption{Transverse deflection at $t=1/1000$.}
\label{fig:deflection}
\end{figure}

\section{Concluding Remarks} \label{sec:conclusions}
We have analyzed the discontinuous Petrov-Galerkin finite element method in the Reissner-Mindlin plate bending problem. The formulation is based on a piecewise polynomial approximation using quadrilateral scalar and vector finite elements of degree $p$ for all quantities of interest (shear stress, bending moment, transverse deflection, rotation). In addition, the resultant tractions and the kinematic variables are approximated on the mesh skeleton by polynomials of degree $p$ and $p+1$, respectively.  

We have showed that the non-standard variational formulation underlying the DPG method is well-posed. Based on that result, we have showed that a discretization where the test functions are approximated in an enriched finite element space of degree $p+3$ is stable as well and leads to optimal order of convergence in the $L_2$ norm for all variables. However, the theoretical stability estimate breaks down at the limit of zero thickness and therefore the final error bound becomes amplified by the factor $t^{-1}$. Our numerical results indicate that some error amplification indeed occurs for the shear force, but the obtained stress values are relatively accurate even on severely distorted meshes. Future work includes formulation of the algorithm for more general geometries and an evaluation of the computational cost and robustness in comparison with other type of formulations. 

\bibliographystyle{acm}
\bibliography{BibTeX-DPG_for_Plates}

\end{document}